\documentclass{amsart}
\usepackage{amsmath}%
\usepackage{amsfonts}%
\usepackage{amssymb}%
\usepackage{graphicx}
\newtheorem{theorem}{Theorem}
\theoremstyle{plain}

\newtheorem{corollary}{Corollary}

\newtheorem{definition}{Definition}
\newtheorem{example}{Example}

\newtheorem{lemma}{Lemma}

\newtheorem{proposition}{Proposition}
\newtheorem{remark}{Remark}

\numberwithin{equation}{section}
\begin{document}
\begin{center}
\vspace*{1.3cm}

\textbf{A UNIFIED APPROACH TO EXTENDED REAL-VALUED FUNCTIONS}

\bigskip

by

\bigskip

PETRA WEIDNER\footnote{HAWK Hochschule Hildesheim/Holz\-minden/G\"ottingen, University of Applied Sciences and Arts, Faculty of Natural Sciences and Technology,
D-37085 G\"ottingen, Germany, {petra.weidner@hawk.de}.}

\bigskip
\bigskip
Research Report\\
Version 6 from June 8, 2018\\
Extended Version of Version 1 from November 16, 2015
\end{center}

\bigskip
\bigskip

\noindent{\small {\textbf{Abstract.}}
Extended real-valued functions are often used in optimization theory, but in different ways for infimum problems and for supremum problems.
We present an approach to extended real-valued functions that works for all types of problems and into which results of convex analysis can be embedded. 
Our approach preserves continuity and the Chebyshev norm when extending a functional to the entire space. The basic idea also works for other spaces than $\overline{\mathbb{R}}$.
Moreover, we illustrate that extended real-valued functions have to be handled in another way than real-valued functions and characterize semicontinuity, convexity, linearity and related properties of such functions.
}

\bigskip


\bigskip


\section{Introduction}

In this report, we present basic notations and properties for functions which attain values in  $\overline{\mathbb{R}} := \mathbb{R}\cup\{-\infty ,+\infty\}$, the {\bf extended set of real numbers}. Of course, such a function may also be real-valued, but infima, suprema, limits, improper integrals or some measure can result in function values $-\infty$ or $+\infty$. Such values can also result from the extension of a function to the entire space or from the addition of an indicator function expressing certain restrictions in convex analysis.

Our approach differs from the usual one in convex analysis presented in \cite{Mor67} and \cite{Roc97}, but results from convex analysis can be embedded into our approach as we will show in Section \ref{s-trans-cx-ua}.  
The main difference is the replacement of $+\infty$ if it is used as a symbol for infeasibility by a symbol which does not belong to $\overline{\mathbb{R}}$. This results in a unified calculus for extended real-valued functions that does not depend on the convexity of the function or on whether the function is to be minimized or maximized. The necessity for an alternative approach to extended real-valued functions came up when studying properties of the function $\varphi_{A,k}$ given by 
\begin{equation}\label{fctal_contr}
\varphi_{A,k} (x):= \inf \{t\in
{\mathbb{R}} \mid x\in tk + A\}
\end{equation}
for some set $A\subset X$ and some $k\in X\setminus\{0\}$ in a linear space X
 under assumptions which do not guarantee that the functional is defined and real-valued on the entire space. Depending on the set $A$ and the vector $k$, $\varphi_{A,k}$ may be convex, concave or neither. Moreover, $\varphi_{A,k}$ may be of interest in minimization and in maximization problems. If the function is convex (e.g. for $X=\mathbb{R}^2$, $A=\{(x_1,x_2)^T\mid x_1 \le 0 ,\; x_2 \le 0\}$, $k=(1,0)^T$), it is of importance for vector minimization problems. If it is concave (e.g. for $X=\mathbb{R}^2$, $A=\{(x_1,x_2)^T\mid 0 \leq x_2 \leq e^{x_1}\}$, $k=(-1,0)^T$), it can be used in economics as a utility function which has to be maximized. 
 
Let us underline that it was not our aim to change the approach in convex analysis, but that we wanted to find a way for extending functions to the entire space in order to use the tools of convex analysis without fixing in advance whether the function would be minimized or maximized. One can easily transfer results from convex analysis, especially from variational analysis, to functions which are given according to our approach. 

Beside the unified calculus, our extension of functions preserves continuity, semicontinuity, suprema and infima of functionals. We get rid of the irritating inequality $\operatorname*{sup}\emptyset < \operatorname*{inf}\emptyset$ and can, e.g., apply the Chebyshev norm approximation immediately to extended real-valued functions which have been extended to the entire space since, in our calculus, the expression
$$\operatorname*{inf}_{y\in Y}\operatorname*{sup}_{x\in X} |f(x)-g(x,y)|$$
makes sense for arbitrary extended real-valued functions $f$, $g$ defined on spaces $X$ and $X\times Y$.

In the following two sections, we will give a short overview about the extended set of real numbers and about the way extended real-valued functionals are handled in convex analysis. Here we will also discuss problems which arise by admitting values $\pm \infty$ and point out more in detail why we  work with a unified approach to extended real-valued functions instead of the standard approach of convex analysis. The unified approach will be described in Section \ref{sec-def-ext}.
Section \ref{s-cont} is devoted to continuity and semicontinuity, whereas Section \ref{s-cx} deals with convexity of extended real-valued functions. In both sections, we also prove basic facts about semicontinuity and convexity which have to be adapted to our framework. Well known results using the epigraph or sublevel sets will appear in a new light. In our proofs, we have to take into consi\-deration that $+\infty$ can belong to the effective domain of the function, that we admit a symbolic function value $\nu\notin\overline{\mathbb{R}}$ and that results of other authors are sometimes based on wrong rules for the calculation in $\overline{\mathbb{R}}$. Since in our theory convexity and concavity have not to be treated separately, affinity and linearity of functions which are not finite-valued can be investigated. This will be done in 
Section \ref{sec-lin-ext}.
Section \ref{s-trans-cx-ua} shows in which way results from convex analysis can be transferred to the unified approach and vice versa and that there exists a one-to-one correspondence between extended real-valued functions in convex analysis and a subset of the extended real-valued functions in the unified approach.
In the last section, we will point out that the unified approach can easily be extended to other spaces than $\overline{\mathbb{R}}$.

From now on, $\mathbb{R}$, $\mathbb{Z}$, $\mathbb{N}$ and $\mathbb{Q}$ will denote the sets of real numbers, of integers, of nonnegative integers and of rational numbers, respectively.
We define $\mathbb{R}_{+}:=\{x\in\mathbb{R}\mid x\ge 0\}$. Linear spaces will always be assumed to be real vector spaces. 
A set $C$ in a linear space is a cone if $\lambda c\in C$ for all $\lambda\in\mathbb{R}_{+}, c\in C$. For a subset $A$ of some linear space, 
$\operatorname*{icr}(A)$ will denote the relative algebraic interior of $A$. 
In a topological space $X$, $\mathcal{N}(x)$ is the set of all neighborhoods of $x\in X$ and $\operatorname*{cl}A$, $\operatorname*{int}A$ and $\operatorname*{bd}A$ denote the closure, the interior and the boundary, respectively, of a subset $A$.
\smallskip

\section{The extended set of real numbers}\label{s-ext-reals}
The extended set of real numbers, $\overline{\mathbb{R}} = \mathbb{R}\cup\{-\infty ,+\infty\}$, is especially used in measure theory and in convex analysis. In measure theory, a set may have the measure $+\infty$. In convex analysis, the following properties of $\overline{\mathbb{R}}$ are of interest:
\begin{itemize}
\item $(\overline{\mathbb{R}},\le )$ is a totally ordered set and each subset of $\overline{\mathbb{R}}$ has an infimum and a supremum,
\item $\overline{\mathbb{R}}$ is a compact Hausdorff space.
\end{itemize}

Let us first recall that in $\overline{\mathbb{R}}$ the following rules are defined for the calculation with $\pm \infty$.
\[ \begin{array}{l} 
\forall y \in \mathbb{R}: \quad -\infty < y < + \infty ,  \\
-(+\infty)=-\infty,\quad  -(-\infty)=+\infty, \\
(+\infty)\cdot (+\infty)=(-\infty)\cdot (-\infty)=+\infty,   \\
(+\infty)\cdot (-\infty)=(-\infty)\cdot (+\infty)=-\infty,  
\end{array} \]
\begin{eqnarray*} 
\forall y\in\mathbb{R} \mbox{ with } y > 0: \quad & y \cdot
(+\infty)  = (+\infty) \cdot y = + \infty, \\
\quad &  y \cdot
( -\infty)  = (-\infty) \cdot y =  - \infty , \\
\forall y\in\mathbb{R} \mbox{ with } y < 0: \quad & y \cdot
(+\infty)  = (+\infty) \cdot y = - \infty ,\\
\quad &  y \cdot
( -\infty)  = (-\infty) \cdot y =  + \infty , \\
\forall y \in {\mathbb{R}} \cup \{+\infty \}: \quad & y + (+\infty )  =  + \infty  + y = + \infty, \\
\forall y \in {\mathbb{R}} \cup \{-\infty \}: \quad & y + (-\infty )  =  - \infty  + y = - \infty. 
\end{eqnarray*}
Moreover, we define
\[0 \cdot ( + \infty ) = ( + \infty ) \cdot 0 = 0 \cdot (- \infty ) = ( - \infty ) \cdot 0 =0.  \]

$\overline{\mathbb{R}}$ is not a real linear space since $+\infty$ and $-\infty$ do not have an inverse element w.r.t. (with regard to) addition and since
the calculation within $\overline{\mathbb{R}}$ is not possible in an analogous way as in $\mathbb{R}$.
\begin{example}\label{ex-extended-real}
The rule $(\lambda + \mu)x=\lambda x+\mu x$ does not hold for arbitrary $\lambda ,\mu\in\mathbb{R}$ and $x\in\{+\infty ,-\infty\}$ since otherwise, e.g.,
$+\infty=(3-2)\cdot (+\infty )=3\cdot(+\infty ) + (-2)\cdot(+\infty )=+\infty +(-\infty )=2\cdot(+\infty ) + (-3)\cdot(+\infty )=(2-3)\cdot (+\infty )=-\infty $.
\end{example}

Though $\overline{\mathbb{R}}$ is totally ordered, equivalent transformations of equations and inequalities in $\mathbb{R}$ do not work in the same way in $\overline{\mathbb{R}}$.
\begin{example}\label{ex-extended-inequ}
In $\overline{\mathbb{R}}$, $a+b=a+c$ does not imply $b=c$, and $a+b\leq a+c$ does not imply $b\leq c$. Choose, e.g.,  $a=-\infty$, $b=0$ and $c=-3$.
\end{example}

But we can equip $\overline{\mathbb{R}}$ with the topology that is generated by a neighborhood base of the Euclidean topology on $\mathbb{R}$ together with the neighborhood bases
$\mathcal{B}(+\infty )=\{\{x\in\overline{\mathbb{R}}\mid x>a\} \mid a\in\mathbb{R}\}$ and $\mathcal{B}(-\infty )=\{\{x\in\overline{\mathbb{R}}\mid x<a\} \mid a\in\mathbb{R}\}$.
Whenever we will consider $\overline{\mathbb{R}}$ as a topological space, this topology will be supposed. Let us note that $\overline{\mathbb{R}}$ is a compact Hausdorff topological space, but not a topological vector space. Nevertheless, the following rules are valid in its topology $\tau$: 
\begin{itemize}
\item[(a)] $\forall O\in\tau\setminus\{\emptyset\} \;\forall \lambda\in\mathbb{R}\setminus \{0\}:\; \lambda O\in\tau$.
\item[(b)] $\forall O\in\tau\setminus\{\emptyset\} \;\forall x\in \mathbb{R}:\; x+O\in\tau$.
\end{itemize}  
\smallskip

\section{Extended real-valued functions in convex analysis}\label{s-fctal-cx-anal}

The approach in this section is part of the basics of convex analysis as they were presented by Moreau \cite{Mor67} and Rockafellar \cite{Roc97} and has become standard when studying minimization problems (see, e.g., \cite{Att84}, \cite{HirLem93}, \cite{aub93}, \cite{Zal:02}, \cite{BorVan10}), which includes the variational analysis for such problems (\cite{RocWet97}, \cite{Mor06a}, \cite{Mor06}). Originally the concept focused on the minimization of real-valued convex functions which are extended to the complete space by the value $+\infty$ and/or for which the barrier $+\infty$ indicates the violation of restrictions. When minimizing a function in the case that feasible solutions exist, the function value $+\infty$ does not alter the minimum. 

The consideration of extended real-valued functions in convex analysis is unilateral. The handling of the function depends on whether it is to be minimized or to be maximized. The usual definitions are given for minimization problems and can be applied to maximization by working with $-\infty$ instead of $+\infty$ where the rule for the addition of $+\infty$ and $-\infty$ has to be changed. The definition of an indicator function also refers to the minimization of the function to which it is added. 

We will now explain the usual approach for a function $f$ that has to be minimized. Even if $f$ would be real-valued it could be extended to a function that is not real-valued any more by the following two reasons.

$f: C\to \overline{\mathbb{R}}$ with $C$ being a subset of some space $X$ can be extended to a function
$g: X\to \overline{\mathbb{R}}$ by defining
\begin{displaymath}
g(x):= \left \{
\begin{array}{ll}
f(x) & \mbox{ if} \quad x \in C, \\
+\infty & \mbox{ if} \quad x \in X \setminus C.
\end{array} \right.
\end{displaymath} 
Functions with values in $\overline{\mathbb{R}}$ also play an important role in optimization theory if the problem
\[
\inf \{f(x) \mid x\in C\}
\]
with $f: X \rightarrow \overline{\mathbb{R}}$ is replaced by
\[
\inf \{g(x) \mid x\in X\}
\]
with $g:=f+\Psi_C$, where $\Psi_C: X \rightarrow \overline{\mathbb{R}}$ is the indicator function of $C$ defined by
\begin{displaymath}
\Psi_C (x):= \left \{
\begin{array}{ll}
0 & \mbox{ if} \quad x \in C, \\
+\infty & \mbox{ if} \quad x \in X \setminus C.
\end{array} \right.
\end{displaymath} 
Here, the rule $+\infty + (-\infty) = +\infty$ has to be applied.

Working with $g$ instead of $f$ offers an elegant way to avoid the necessity for investigating technical details of $C$ in advance.
E.g. when studying the functional $\varphi_{A,k}$ given by equation (\ref{fctal_contr}) one can simply write $\varphi_{A,k} : X \rightarrow \overline{\mathbb{R}}$ using the definition $\inf \emptyset = + \infty$.

Obviously, the value $+\infty$ serves as a symbol for a function value at points which are not feasible for the problem, and if a function value $+\infty$ exists since, e.g., the function is defined via a supremum, the related argument of the function is handled as if it would not belong to the domain of interest. 

Consequently, the (effective) domain of a function $g: X\to \overline{\mathbb{R}}$ is defined as 
$\{x\in X\mid g(x)\in\mathbb{R}\cup \{ -\infty\}\}$. 

Based on this classical approach, useful tools for convex optimization have been developed, but we have to point out that extended real-valued functions have to be handled carefully. 
One has to take into consideration that the familiar laws of arithmetic are not valid in full extent on $\overline{\mathbb{R}}$ (cp. Example \ref{ex-extended-real}) and that equations as well as inequalities cannot be handled in the same way as for real numbers (cp. Example \ref{ex-extended-inequ}). 
Let us e.g. recall that a function $f: C \rightarrow \mathbb{R}$, $C$ being a convex subset of a linear space $X$, is said to be convex on $C$
if
\begin{equation*}         
f(\lambda x^1+ (1-\lambda) x^2) \leq \lambda f(x^1) + (1-\lambda ) f(x^2)
\end{equation*}
holds for all $x^1, x^2\in C$ and $\lambda\in (0,1)$.
Convexity of functions with values in $\overline{\mathbb{R}}$ has to be defined in an alternative way, e.g. via convexity of the epigraph of $f$ (cp. e.g. \cite{Roc97}).

We will now point out drawbacks of the approach in convex analysis which will not appear in the unified approach we are going to present in the next section.
\begin{itemize}
\item[(a)]
The way an extended real-valued function in convex analysis is handled depends on whether it has to be minimized or to be maximized.
Supremum problems are studied in an analogous way as the infimum problems above, but with $-\infty$ replacing $+\infty$ and vice versa. One consequence is the different definition of $+\infty + (-\infty )$ in both frameworks.
The so-called inf-addition
\[ +\infty + (-\infty ) = (-\infty ) + (+\infty ) = +\infty \]
is applied for minimization problems, whereas the sup-addition for maximization problems is given by 
\[ +\infty + (-\infty ) = (-\infty ) + (+\infty ) = -\infty .\]
Moreau \cite{Mor67} introduced both kinds of addition and, in connection with this, two different addition operators on $\overline{\mathbb{R}}$. 
Note the following consequence of the inf-addition for $a,b,c\in\overline{\mathbb{R}}$:
\begin{equation*}
a\leq b+c\mbox{ implies } -a\geq -(b+c) \mbox{, but not } -a\geq -b-c
\end{equation*}
since, e.g., $3\leq (+\infty )+(-\infty )$, but $-3\not\geq (-\infty )+(+\infty )$.
Nevertheless,
\begin{equation*}
a\leq b+c\implies -a\geq -b-c
\end{equation*}
if $\{-\infty ,+\infty \}\not= \{ b,c \}$.
\item[(b)]
The extension of a convex function $f: C \rightarrow \overline{\mathbb{R}}$, $C$ being a proper convex subset of some topological vector space $X$, to $g: X\to \overline{\mathbb{R}}$ by adding the function value $+\infty$ outside $C$
can destroy the continuity and even the lower semicontinuity of $f$.
\begin{example}
Consider the functional $\varphi : \mathbb{R}\to \overline{\mathbb{R}}$ given by
\[ \varphi(x)=\left\{
\begin{array}{r@{\quad\mbox{ if }\quad}l}
x & x>0,\\
+\infty & x\leq 0.
\end{array}
\right.
\]
$\varphi$ is continuous on $\{x\in \mathbb{R}\mid \varphi(x)\in\mathbb{R}\cup \{-\infty\}\}$, but not lower semiconti\-nuous on $\mathbb{R}$.
\end{example}
\item[(c)] Multiplication of functions with real numbers and subtraction of functions are only possible under certain assumptions \cite[p.389]{HirLem93}.
\item[(d)] In connection with extended real-valued functions in convex analysis, the definitions $\operatorname*{inf}\emptyset =+\infty$ and $\operatorname*{sup}\emptyset =-\infty$ are given. Hence $\operatorname*{sup}\emptyset < \operatorname*{inf}\emptyset$.
\item[(e)]
If a function is extended to the entire space by $+\infty$ or $-\infty$ it cannot be handled like the original function any more. This is obviously the case since the way of extension depends on the purpose. The extended function does also not have a finite Chebyshev norm and the Chebyshev norm approximation
$$\operatorname*{inf}_{y\in Y}\operatorname*{sup}_{x\in X} |f(x)-g(x,y)|$$
cannot be applied.
\item[(f)]
In some optimization problems a function cannot simply be extended to the entire space using only one of the values $+\infty$ or $-\infty$. Consider the problem
\begin{equation}\label{eq-infsupL}
\operatorname*{inf}_{x\in X_0}\operatorname*{sup}_{y\in Y_0} L(x,y),
\end{equation}
where $Y_0$ is a proper subset of some space $Y$, $X_0$ is a proper subset of some space $X$ and $L$ is real-valued. In such a case $L$ is extended to a function $\ell$ which is defined on $X\times Y$
and for which the problem 
$$\operatorname*{inf}_{x\in X}\operatorname*{sup}_{y\in Y} \ell (x,y)$$
is equivalent to problem (\ref{eq-infsupL}) in the following way \cite[Example 11.52]{RocWet97}:
\[ \ell (x,y)=\left\{
\begin{array}{r@{\quad\mbox{ if }\quad}l}
L(x,y) & x\in X_0, y\in Y_0,\\
-\infty & x\in X_0, y\in Y\setminus Y_0,\\
+\infty & x\in X\setminus X_0.
\end{array}
\right.
\] 
\item[\rm (g)] The previous item illustrates that the unilateral approach in convex analysis becomes complicated when infimum problems and supremum problems are combined. Duality theory usually works with such combined problems. Then one has to be very careful in using the notion of the effective domain or adding $+\infty$ and $-\infty$. The usual indicator function in convex analysis should not be added to functions which have to be maximized.
\end{itemize}
\smallskip

\section{Unified approach to extended real-valued functions}\label{sec-def-ext}
We introduce the symbol $\nu$ as a function value in arguments which are not feasible otherwise. Keeping this in mind, definitions of notions and properties for extended real-valued functions will emerge in a natural way.
For sets $A\subseteq\overline{\mathbb{R}}$ we define $A_{\nu}:=A\cup\{\nu\}$. 

On subsets of a space $X$ $(X\not= \emptyset )$, we consider functions which can take values in $\overline{\mathbb{R}}$.  
If a function $\varphi$ is defined on a subset $X_0\subseteq X$, we extend the range of definition to the entire space $X$ by defining $\varphi(x):=\nu$ for all $x\in X\setminus X_{0}$. This yields a function
$\varphi : X\to \overline{\mathbb{R}}_{\nu}$. We call functions with values in $\overline{\mathbb{R}}_{\nu}$ {\bf extended real-valued functions}
and often refer to them simply as functionals.

\begin{definition}\label{d-eff-part}
Consider an extended real-valued function $\varphi : X\to \overline{\mathbb{R}}_{\nu}$ on some non\-empty set $X$.
We define its \textbf{(effective) domain} as 
\[ \operatorname*{dom}\varphi:=\{x\in X\mid \varphi(x)\in\overline{\mathbb{R}}\}. \]
$\varphi$ is \textbf{trivial} if
$\operatorname*{dom}\varphi =\emptyset$.
$\varphi$ is called \textbf{finite-valued} on $X_0\subseteq X$ if $\varphi(x)\in\mathbb{R}$ for all $x\in X_0$. 
$\varphi$ is said to be \textbf{finite-valued} if $\varphi$ is finite-valued on $X$.\\ 
$\varphi$ is called \textbf{proper} if
$\varphi$ is nontrivial and finite-valued on $\operatorname*{dom}\varphi$. Otherwise, it is said to be \textbf{improper}.\\ 
For $\varphi$ being nontrivial, we introduce $\varphi^{\operatorname*{eff}} : \operatorname*{dom}\varphi \rightarrow \overline{\mathbb{R}}$ by $\varphi^{\operatorname*{eff}}(x)=\varphi(x)$ for all $x\in\operatorname*{dom}\varphi $ and call it the 
\textbf{effective part} of $\varphi$.

\end{definition}

\begin{remark}
Our definition of the effective domain essentially differs from the usual definition in convex analysis (see Section \ref{s-fctal-cx-anal}) since we use $\nu $ instead of $+\infty $ and admit $+\infty $ to be a function value that comes into existence, e.g., by defining the function as some supremum, and may be of interest in problems which depend on the function. 
\end{remark}

The finite-valued functionals are just the functions $\varphi : X\to \mathbb{R}$ and are just the proper functionals with 
$\operatorname*{dom}\varphi = X$. For finite-valued functionals, $\varphi^{\operatorname*{eff}}=\varphi$. The proper functions are the nontrivial functions $\varphi : X\to \mathbb{R}_{\nu}$.

Each function and operation applied to $\nu$ has to result in $\nu$ if the function maps into $\overline{\mathbb{R}}_{\nu}$, and to result in the empty set if the function value should be a set. \\
In $\overline{\mathbb{R}}_{\nu}$, the following rules are defined for the calculation with $\nu$:
\[ \begin{array}{l} 
\forall y \in \overline{\mathbb{R}}: \quad \nu \not\le y,   \quad   y \not\le \nu,    \\
- \nu=\nu, \\
+ \infty  + (- \infty) =  - \infty  + (+ \infty)  =\nu, \\
\forall y\in\overline{\mathbb{R}}_{\nu}: \quad y \cdot \nu  =  \nu \cdot y  =\nu,  \\ 
\forall y \in \overline{\mathbb{R}}_{\nu}: \quad y + \nu  =  \nu  + y = \nu. 
\end{array} \]

The above definitions extend the binary relations $<$, $\le$, $>$, $\ge$, $=$, the binary operations $+$ and $\cdot$ as well as the unary operation $-$ to $\overline{\mathbb{R}}_{\nu}$.
Note that the relations $\not<$, $\not\leq$, $\not>$ and $\not\geq$ coincide with $\geq$, $>$, $\leq$ and $<$, respectively, on $\overline{\mathbb{R}}$, but not on $\overline{\mathbb{R}}_{\nu}$. 
Take also into consideration that we have defined $-y$ for all $y\in\overline{\mathbb{R}}_{\nu}$ as a unary operation, but that $y+(-y)=\nu\not= 0$ for $y\in\{-\infty,+\infty,\nu\}$.
With this unary operation, we can define the subtraction on $\overline{\mathbb{R}}_{\nu}$ by
$y_1-y_2:=y_1+(-y_2)$ for all $y_1,y_2\in\overline{\mathbb{R}}_{\nu}$.
Moreover, we define
\[ |+ \infty|=|- \infty|=+\infty \mbox{  and  } |\nu |=\nu. \]

Now we can transfer these definitions to functions.\\
For some nonempty set $X$, functions $f,g: X\to \overline{\mathbb{R}}_{\nu}$ and $\lambda\in\overline{\mathbb{R}}$, we define 
\[ \begin{array}{rlll} 
-f: & X\to \overline{\mathbb{R}}_{\nu} & \mbox{ by } & \quad(-f)(x)=-f(x)\mbox{ for all }x\in X, \\
\lambda\cdot f: & X\to \overline{\mathbb{R}}_{\nu} & \mbox{ by } & \quad(\lambda\cdot f)(x)=\lambda\cdot f(x)\mbox{ for all }x\in X, \\
f+g: & X\to \overline{\mathbb{R}}_{\nu} & \mbox{ by } & \quad (f+g)(x)=f(x)+g(x)\mbox{ for all }x\in X,  \\
f+\lambda : & X\to \overline{\mathbb{R}}_{\nu} & \mbox{ by } & \quad (f+\lambda )(x)=f(x)+\lambda\mbox{ for all }x\in X,  \\
\lambda +f: & X\to \overline{\mathbb{R}}_{\nu} & \mbox{ by } & \quad \lambda +f=f+\lambda. 
\end{array} \]
Note that\\ 
$\operatorname*{dom}(f+g) = (\operatorname*{dom}f\cap\operatorname*{dom}g)\setminus\{x\in X\mid f(x)\in\{-\infty,+\infty\}\;\mbox{ and }\; g(x)=-f(x)\},$\\
$\operatorname*{dom}(f+\lambda )  =  \operatorname*{dom}(\lambda +f)=\operatorname*{dom}f\setminus\{x\in X\mid f(x)=-\lambda\}\mbox{  if }\lambda\in\{-\infty,+\infty\}.$

Extended real-valued functions can be used as indicator functions of sets.
\begin{definition}\label{def-indic}
For a subset $A$ of a space $X$, the \textbf{indicator function} of $A$ is the function
$\iota_A : X\rightarrow\overline{\mathbb{R}}_{\nu}$ defined by
\begin{displaymath}
\iota_A (x):= \left \{
\begin{array}{ll}
0 & \mbox{ if} \quad x \in A, \\
\nu & \mbox{ if} \quad x \in X \setminus A.
\end{array} \right.
\end{displaymath} 
\end{definition}

\begin{remark}
The usual definition of indicator functions in convex analysis uses $+\infty $ instead of $\nu $.
The indicator functions in measure theory are defined with values $1$ and $0$ instead of $0$ and $\nu$, respectively.
\end{remark}

The extension of functions to the whole space is not the only reason for the importance of extended real-valued functions. Such functions also yield the possibility to replace an optimization problem with side conditions by a free optimization problem. 
If we are looking for optimal values of some function $f: X\to \overline{\mathbb{R}}_{\nu}$ on the set $A\subset X$, then this problem is equivalent to
the calculation of optimal values of the function $g: X\to \overline{\mathbb{R}}_{\nu}$ if $g:=f+\iota_A$.

Since infima and suprema play a central role in optimization theory, we now extend these notions to $\overline{\mathbb{R}}_{\nu}$.

\begin{definition}
A set $A\subseteq\overline{\mathbb{R}}_{\nu}$ is called \textbf{trivial} if $A\subseteq\{\nu \}$.
$b\in\overline{\mathbb{R}}$ is a \textbf{lower bound} of a nontrivial set $A\subseteq\overline{\mathbb{R}}_{\nu}$ if $a\not < b$ for all $a\in A$.
$b\in\overline{\mathbb{R}}$ is an \textbf{upper bound} of a nontrivial set $A\subseteq\overline{\mathbb{R}}_{\nu}$ if $a\not > b$ for all $a\in A$.
A nontrivial set $A\subseteq\overline{\mathbb{R}}_{\nu}$ is \textbf{bounded below} or \textbf{bounded above} if there exists some real lower bound or some real upper bound, respectively, of $A$. It is \textbf{bounded} if it is bounded below and bounded above.
The \textbf{infimum} $\operatorname*{inf}A$ and the \textbf{supremum} $\operatorname*{sup}A$ of $A\subseteq\overline{\mathbb{R}}_{\nu}$ are defined by
\begin{displaymath}
\operatorname*{inf}A:= \left \{
\begin{array}{cl}
\nu & \mbox{ if } A \mbox{ is trivial}, \\
\mbox{ the largest lower bound of } A & \mbox{ otherwise},
\end{array} \right.
\end{displaymath}
\begin{displaymath}
\operatorname*{sup}A:= \left \{
\begin{array}{cl}
\nu & \mbox{ if } A \mbox{ is trivial}, \\
\mbox{ the smallest upper bound of } A & \mbox{ otherwise}.
\end{array} \right.
\end{displaymath} 
We say that $A$ has a \textbf{minimum} $\min A$ if $\operatorname*{inf}A\in A\cap\mathbb{R}$ and define, in this case, $\min A:=\operatorname*{inf}A$.
We say that $A$ has a \textbf{maximum} $\max A$ if $\operatorname*{sup}A\in A\cap\mathbb{R}$ and define, in this case, $\max A:=\operatorname*{sup}A$.
\end{definition}

Obviously, $-\infty$ is a lower bound and $+\infty$ is an upper bound of each nontrivial set in $\overline{\mathbb{R}}_{\nu}$. If $A\cap \overline{\mathbb{R}}=\{+\infty\}$, then $+\infty$ is also a lower bound of $A$. We get $\operatorname*{inf}\emptyset =\operatorname*{sup}\emptyset =\nu$. 
For nonempty sets in $\overline{\mathbb{R}}$, the above notions coincide with the usual ones.

Now we can define the Minkowski functional in our framework. 
\begin{definition}
Consider some subset $A$ of a linear space $X$ with $0\in A$.
The \textbf{Minkowski functional} $p_A:X\rightarrow\overline{\mathbb{R}}_{\nu}$ of $A$ is defined by
\[ p_A(x)=\operatorname*{inf}\{\lambda >0\mid x\in\lambda A\}.\]
\end{definition}

Then the Minkowski functional of a cone is just its indicator function \cite{Mor67}.

Let us adapt notions which are needed for dealing with differentiability and other properties of functions to extended real-valued functions. 

\begin{definition}
Assume that $X$ is an arbitrary set and $\varphi : X\to \overline{\mathbb{R}}_{\nu}$.\\
The \textbf{infimum} and the \textbf{supremum} of $\varphi$ (on $X$) are defined by $\operatorname*{inf}_{x\in X}\varphi (x) :=\operatorname*{inf}\{\varphi (x)\mid x\in X\}$ and $\operatorname*{sup}_{x\in X}\varphi (x) :=\operatorname*{sup}\{\varphi (x)\mid x\in X\}$, respectively.\\
If $\{\varphi (x)\mid x\in X\}$ has a minimum or maximum, we say that $\varphi $ attains a \textbf{(global) minimum} or a \textbf{(global) maximum}, respectively, on $X$ and denote it by\\
$\min_{x\in X}\varphi(x)$ or $\max_{x\in X}\varphi(x)$, respectively.\\
Suppose now $\operatorname*{dom}\varphi\not=\emptyset$.\\
$t\in\overline{\mathbb{R}}$ is called an \textbf{upper bound} or a \textbf{lower bound} of $\varphi$ (on $X$) if it is an upper or lower bound of $\{\varphi(x)\mid x\in X\}$, respectively.
$\varphi$ is called \textbf{bounded above}, \textbf{bounded below} or \textbf{bounded} (on $X$) if $\{\varphi(x)\mid x\in X\}$ is bounded above, bounded below or bounded, respectively.
\end{definition}

Take into consideration that each nontrivial extended real-valued functional has an upper bound and a lower bound, but that it is only bounded above or bounded below if there exists some real upper bound or lower bound, respectively. For functions with values in $\overline{\mathbb{R}}$ only, the above definition and the next one are compatible with the usual notions. 

\begin{definition}
Suppose that $X$ is a topological space and $\varphi : X\to \overline{\mathbb{R}}_{\nu}$.\\
If $x^0\in\operatorname*{cl}\operatorname*{dom}\varphi$ and if there exists some $g\in\overline{\mathbb{R}}$ such that for each neighborhood $V$ of $g$ there exists some neighborhood $U$ of $x^0$ with $\varphi (x)\in V$ for all $x\in U\cap\operatorname*{dom}\varphi$, then $g$ is said to be the limit of $\varphi$ at $x^0$. If $x^0\in X\setminus \operatorname*{cl}\operatorname*{dom}\varphi$ or if there does not exist some limit of $\varphi$ at $x^0$ in $\overline{\mathbb{R}}$,
then the limit of $\varphi$ at $x^0$ is defined as $\nu $. The \textbf{limit} of $\varphi$ at $x^0$ is denoted by $\operatorname*{lim}_{x\to x^0}\varphi (x)$.
\end{definition}

For derivatives of functions, we also need limits of sequences.
\begin{definition}\label{d-lim_nu}
A function mapping $\mathbb{N}$ into $\overline{\mathbb{R}}$ is called an \textbf{(extended real-valued) sequence}.
We will denote such a sequence by $(a_n)$, where $a_n$ is the function value of $n$ for each $n\in\mathbb{N}$.
The sequence is \textbf{bounded} if the mapping is bounded.
A value $g\in\overline{\mathbb{R}}$ is said to be a \textbf{limit} of $(a_n)$, denoted as $\operatorname*{lim}_{n\to +\infty} a_n$,
if for each neighborhood $U$ of $g$ there exists some $n_0\in\mathbb{N}$ such that $a_n\in U$ for each $n\in\mathbb{N}$ with $n>n_0$. 
If the sequence does not have any limit in $\overline{\mathbb{R}}$, then $\nu $ is defined to be the limit of the sequence.\\
A value $g\in\overline{\mathbb{R}}$ is a \textbf{cluster point} of $(a_n)$ if each neighborhood of $g$ contains $a_n$ for an infinite number of elements $n\in \mathbb{N}$. 
Let $C$ denote the set of cluster points of $(a_n)$. $\operatorname*{sup} C$ is called the \textbf{limit superior} $\;\operatorname*{lim}\operatorname*{sup}a_n$,
$\operatorname*{inf} C$ is called the \textbf{limit inferior} $\;\operatorname*{lim}\operatorname*{inf}a_n\;$ of $(a_n)$. 
\end{definition}

For real-valued sequences, the above notions work as usual.

Example \ref{ex-extended-real} demonstrates that linear combinations of improper functionals have to be used carefully.
This requires alternative definitions of important properties like convexity of functions. The alternative way of describing such a property sometimes uses the epigraph or the hypograph of a function.

\begin{definition}
Assume that  $X$ is a non\-empty set and $\varphi : X\to \overline{\mathbb{R}}_{\nu}$.
The \textbf{epigraph} of $\varphi$ is defined by
\begin{equation*}
\operatorname*{epi}\varphi:=\{(x,t)\in X\times \mathbb{R}\mid \varphi(x)\le t\}.
\end{equation*}
The \textbf{hypograph} or \textbf{subgraph} of $\varphi$ is the set
\begin{equation*}
\operatorname*{hypo}\varphi:=\{(x,t)\in X\times \mathbb{R}\mid \varphi(x)\ge t\}.
\end{equation*}
\end{definition}

Whenever $X$ is a topological space, the investigation of topological properties of the epigraph or of the hypograph refers to the space $X\times \mathbb{R}$.

Let us mention some immediate consequences of the definition.
\begin{lemma}\label{l-epihypo}
Let $X$ be a non\-empty set and $\varphi : X\to \overline{\mathbb{R}}_{\nu}$.
\begin{itemize}
\item[(a)] $(x,t)\in\operatorname*{epi}\varphi$ if and only if $(x,-t)\in \operatorname*{hypo}(-\varphi)$.
\item[(b)] For each $x\in X$, $\varphi (x)= -\infty $ if and only if $\{x\}\times \mathbb{R}\subseteq \operatorname*{epi}\varphi$.
In this case, $(\{x\}\times \mathbb{R})\cap \operatorname*{hypo}\varphi =\emptyset$.
\item[(c)] For each $x\in X$, $\varphi (x)= +\infty $ if and only if $\{x\}\times \mathbb{R}\subseteq \operatorname*{hypo}\varphi$.
In this case, $(\{x\}\times \mathbb{R})\cap \operatorname*{epi}\varphi =\emptyset$.
\end{itemize}
\end{lemma}

Part (a) of this lemma yields:
\begin{corollary}
Let $X$ be a linear space and $\varphi : X\to \overline{\mathbb{R}}_{\nu}$.
\begin{itemize}
\item[(a)] $\operatorname*{hypo}\varphi +\operatorname*{hypo}\varphi\subseteq\operatorname*{hypo}\varphi\iff\operatorname*{epi}(-\varphi) + \operatorname*{epi}(-\varphi)\subseteq\operatorname*{epi}(-\varphi)$.
\item[(b)] $\operatorname*{hypo}\varphi$ is convex $\iff$ $\operatorname*{epi}(-\varphi)$ is convex.
\item[(c)] $\operatorname*{hypo}\varphi$ is a cone $\iff$ $\operatorname*{epi}(-\varphi)$ is a cone.
\end{itemize}
\end{corollary}

\begin{corollary}\label{hypo-closed}
Let $X$ be a topological space and $\varphi : X\to \overline{\mathbb{R}}_{\nu}$. Then \\
$\operatorname*{hypo}\varphi$ is closed $\iff$ $\operatorname*{epi}(-\varphi)$ is closed.
\end{corollary}
\smallskip

\section{Continuity and Semicontinuity}\label{s-cont}
We now extend the definition of continuity to functionals $\varphi: X\rightarrow \overline{\mathbb{R}}_{\nu }$.

\begin{definition}\label{d-cont-ext-real}
Let $X$ be a topological space, $\varphi: X\rightarrow \overline{\mathbb{R}}_{\nu }$.
$\varphi$ is \textbf{continuous} at $x\in X$ if $x\in\operatorname*{dom}\varphi $ and $\varphi^{\operatorname*{eff}}$ is continuous at $x$. 
$\varphi$ is \textbf{continuous} on the non\-empty set $D\subseteq X$ if $D\subseteq\operatorname*{dom}\varphi $ and $\varphi^{\operatorname*{eff}}$ is continuous on $D$.
If $\operatorname*{dom}\varphi $ is nonempty and closed, then we call $\varphi$ a \textbf{continuous functional} if $\varphi$ is continuous on $\operatorname*{dom}\varphi$. 
\end{definition}

\begin{example}\label{ex-extended-func}
Consider the functional $\varphi : \mathbb{R}\to \overline{\mathbb{R}}_{\nu}$ given by
\[ \varphi(x)=\left\{
\begin{array}{r@{\quad\mbox{ if }\quad}l}
\operatorname*{tan}(x) & -\frac{\pi}{2}<x<\frac{\pi}{2},\\
-\infty &  x=-\frac{\pi}{2}, \\
+\infty &  x=\frac{\pi}{2}, \\
\nu & x<-\frac{\pi}{2} \;\mbox{ or }\; x>\frac{\pi}{2}.
\end{array}
\right.
\]

$ \operatorname*{dom}\varphi=\{x\in \mathbb{R}\mid -\frac{\pi}{2} \le x \le +\frac{\pi}{2}\} $ is closed, and
$\varphi$ is a continuous functional.
\end{example}

Note that the previous example also illustrates that there exists a one-to-one-correspondence between the closed interval $[-\frac{\pi}{2},\frac{\pi}{2}]$ and $\overline{\mathbb{R}}$ and also a one-to-one-correspondence between the open sets induced on $[-\frac{\pi}{2},\frac{\pi}{2}]$ by the Euclidean topology and the open sets in $\overline{\mathbb{R}}$. Indeed, it is well known that $\overline{\mathbb{R}}$ can be considered as a two-point compactification of $\mathbb{R}$ since (as  previously mentioned) $\overline{\mathbb{R}}$ is a compact topological space.

Functionals which are continuous on $\operatorname*{dom}\varphi$ can often be extended in such a way that the domain of the extended functional is closed and the extended functional is continuous.

\begin{example}\label{ex-betrag}
Let $\varphi : \mathbb{R}\to \overline{\mathbb{R}}_{\nu}$ be given by
\[ \varphi(x)=\left\{
\begin{array}{r@{\quad\mbox{ if }\quad}l}
\frac{1}{|x|} & x\not= 0,\\
\nu & x=0.
\end{array}
\right.
\]
$\varphi$ is continuous on $\operatorname*{dom}\varphi$, but $\operatorname*{dom}\varphi$ is not closed. $\varphi$ can be extended to a continuous functional with $\operatorname*{dom}\varphi=\mathbb{R}$ by replacing the symbolic function value $\nu$ in $x=0$ by $+\infty$.
\end{example}

Of course, not each functional which is continuous on its effective domain can be extended to a continuous function, e.g., if there exists a jump between two parts of the effective domain. But in contrast to the usual approach in convex analysis, semicontinuity and continuity of a functional cannot be destroyed by extending the function to the entire space using $\nu $.

The definition of continuity implies:
\begin{lemma}\label{l-cont-anders}
Let $X$ be a topological space and $\varphi :X \rightarrow \overline{\mathbb{R}}_{\nu }$. 
\begin{itemize}
\item[(a)] If $\varphi$ is continuous at $x^0\in\operatorname*{dom}\varphi$, then the functions 
$\lambda\varphi$ and $\varphi + \lambda$ with $\lambda\in\mathbb{R}$ are continuous at $x^0$.
\item[(b)] If $X$ is a topological vector space, $x^0, \overline{x}\in X$, $x^0-\overline{x} \in\operatorname*{dom}\varphi$ and $\varphi$ is continuous at $x^0-\overline{x}$, then
$g:X \rightarrow \overline{\mathbb{R}}_{\nu }$ defined by $g(x)=\varphi (x-\overline{x})$ is continuous at $x^0$.
\item[(c)] If $X$ is a topological vector space, $x^0\in X$, $\lambda\in\mathbb{R}$, $\lambda x^0 \in\operatorname*{dom}\varphi$ and $\varphi$ is continuous at $\lambda x^0$, then
$g:X \rightarrow \overline{\mathbb{R}}_{\nu }$ defined by $g(x)=\varphi (\lambda x)$ is continuous at $x^0$.
\end{itemize}
\end{lemma}

The product of continuous extended real-valued functions is not necessarily continuous. 

\begin{example}
Define $f, g : \mathbb{R}\to \overline{\mathbb{R}}$ by
\[ f(x)=\left\{
\begin{array}{l@{\quad\mbox{ if }\quad}l}
\frac{1}{|x|} & x\not= 0,\\
+\infty & x=0,
\end{array}
\right.
\]
\[ g(x):=|x|. \]
Then 
\[ h(x):=g(x)\cdot f(x)=\left\{
\begin{array}{l@{\quad\mbox{ if }\quad}l}
1 & x\not= 0,\\
0\cdot (+\infty)=0 & x=0.
\end{array}
\right.
\]
$\operatorname*{dom}f=\operatorname*{dom}g=\mathbb{R}$. $f$ and $g$ are continuous functions, but $h$ is not continuous at $x=0$.
\end{example}
Let us now introduce semicontinuity of functionals, first for functionals with values in $\overline{\mathbb{R}}$ only.
\begin{definition}
Consider a topological space $X$, a nonempty set $D\subseteq X$ and a function $\varphi :D \rightarrow \overline{\mathbb{R}}$.\\
$\varphi$ is called \textbf{lower semicontinuous} at $x^{0} \in D$ if $\varphi(x^{0})=-\infty$ or
for each $h\in\mathbb{R}$ with $h < \varphi(x^{0})$ there exists some $U\in\mathcal{N}(x^{0})$ such that $\varphi(x) > h$ for all $x\in U\cap D$.\\
$\varphi$ is called \textbf{upper semicontinuous} at $x^{0} \in D$ if $\varphi(x^{0})=+\infty$ or
for each $h\in\mathbb{R}$ with $h > \varphi(x^{0})$ there exists some $U\in\mathcal{N}(x^{0})$ such that $\varphi(x) < h$ for all $x\in U\cap D$.\\
$\varphi$ is called a \textbf{lower semicontinuous function} or an \textbf{upper semicontinuous function} (on $D$) if $\varphi$ is lower semicontinuous or upper semicontinuous, respectively, at each $x\in D$.
\end{definition}

We now extend this definition to functionals with values in $\overline{\mathbb{R}}_{\nu }$.
\begin{definition}\label{dsemicon}
Let $X$  be a topological space, $\varphi: X\rightarrow \overline{\mathbb{R}}_{\nu }$.\\
$\varphi$ is \textbf{lower semicontinuous} or \textbf{upper semicontinuous} at $x\in X$ if $x\in\operatorname*{dom}\varphi $ and $\varphi^{\operatorname*{eff}}$ is lower semicontinuous or upper semicontinuous, respectively, at $x$. 
$\varphi$ is \textbf{lower semicontinuous} or \textbf{upper semicontinuous} on the nonempty set $D\subseteq X$ if $D\subseteq\operatorname*{dom}\varphi $ and $\varphi^{\operatorname*{eff}}$ is lower semicontinuous or upper semicontinuous, respectively, on $D$.
If $\operatorname*{dom}\varphi$ is nonempty and closed, then we call $\varphi$ a \textbf{lower semicontinuous function} or an \textbf{upper semicontinuous function} if $\varphi$ is lower semicontinuous or upper semicontinuous, respectively, on $\operatorname*{dom}\varphi$. 
\end{definition}

Note that each functional which can attain only one value $c\in\overline{\mathbb{R}}$ on $\operatorname*{dom}\varphi$ is
lower semicontinuous, upper semicontinuous and continuous on $\operatorname*{dom}\varphi$.

Immediately from the definitions we get the following statements.
\begin{proposition}\label{p-usc-lsc}
Let $X$ be a topological space, $\varphi: X\rightarrow \overline{\mathbb{R}}_{\nu }$, $x\in \operatorname*{dom}\varphi$.
\begin{itemize}
\item[(a)] $\varphi$ is upper semicontinuous at $x$ iff $-\varphi$ is lower semicontinuous at $x$.
\item[(b)] $\varphi$ is continuous at $x$ iff $\varphi$ is lower semicontinuous and upper semicontinuous at $x$.
\end{itemize}
\end{proposition}

Definition \ref{dsemicon} implies:
\begin{lemma}\label{l-semicon-anders}
Let $X$ be a topological space and $\varphi :X \rightarrow \overline{\mathbb{R}}_{\nu }$. 
\begin{itemize}
\item[(a)] If $\varphi$ is lower semicontinuous at $x^0\in\operatorname*{dom}\varphi$, then the functions 
$\lambda\varphi$ with $\lambda \in\mathbb{R}_{+}$ and $\varphi + c$ with $c\in\mathbb{R}$ are lower semicontinuous at $x^0$.
\item[(b)] If $X$ is a topological vector space, $x^0, \overline{x}\in X$, $x^0-\overline{x} \in\operatorname*{dom}\varphi$ and $\varphi$ is lower semicontinuous at $x^0-\overline{x}$, then
$g:X \rightarrow \overline{\mathbb{R}}_{\nu }$ defined by $g(x)=\varphi (x-\overline{x})$ is lower semicontinuous at $x^0$.
\item[(c)] If $X$ is a topological vector space, $x^0\in X$, $\lambda\in\mathbb{R}$, $\lambda x^0 \in\operatorname*{dom}\varphi$ and $\varphi$ is lower semicontinuous at $\lambda x^0$, then
$g:X \rightarrow \overline{\mathbb{R}}_{\nu }$ defined by $g(x)=\varphi (\lambda x)$ is lower semicontinuous at $x^0$.
\end{itemize}
The same statements hold for upper semicontinuity.
\end{lemma}

\begin{corollary}\label{c-semi-induced}
Let $(X,\tau )$ be a topological space, $\varphi: X\rightarrow \overline{\mathbb{R}}_{\nu }$, $D\subseteq\operatorname*{dom}\varphi$ with $D\not=\emptyset$.
Then the following properties of $\varphi$ are equivalent to each other:
\begin{itemize}
\item[(a)] $\varphi$ is lower semicontinuous on $D$.
\item[(b)] Each set $\{x\in D\mid \varphi (x)>t\}$, $t\in\mathbb{R}$, is open w.r.t. the topology induced by $\tau$ on D.
\item[(c)] Each set $\{x\in D\mid \varphi (x)\le t\}$, $t\in\mathbb{R}$, is closed w.r.t. the topology induced by $\tau$ on D.
\end{itemize}
\end{corollary}

The next propositions will connect semicontinuity of a functional $\varphi$ with topological properties of the sublevel sets 
and of the sets $\{x\in X\mid \varphi (x)>t\}$ without referring explicitly 
to some induced topology.
\begin{definition}
Let $R$ denote some binary relation on $\overline{\mathbb{R}}_{\nu }$, $X$ be a nonempty set, $\varphi: X\rightarrow \overline{\mathbb{R}}_{\nu }$.
Then we define $\operatorname*{lev}_{\varphi,R}(t):= \{x\in X \mid \varphi (x) R t\}$ for $t\in\mathbb{R}$.
\end{definition} 

\begin{proposition}\label{prop-semicon-levels-open}
Let  $\varphi: X\rightarrow \overline{\mathbb{R}}_{\nu }$ be a nontrivial function on a topological space $X$.
\begin{itemize}
\item[(a)]  $\varphi$ is lower semicontinuous on $\operatorname*{dom}\varphi$ if the sets $\operatorname*{lev}_{\varphi,>}(t)$ are open for all 
$t\in \mathbb{R}$. 
\item[(b)] Assume that $\operatorname*{dom}\varphi$ is open. Then \\
$\varphi$ is lower semicontinuous on $\operatorname*{dom}\varphi$ if and only if the sets $\operatorname*{lev}_{\varphi,>}(t)$ are open for all $t\in \mathbb{R}$. 
\item[(c)] If $\varphi$ is bounded below, then:\\
The sets $\operatorname*{lev}_{\varphi,>}(t)$ 
are open for all $t\in \mathbb{R}$ if and only if $\operatorname*{dom}\varphi$ is open and $\varphi$ is lower semicontinuous on $\operatorname*{dom}\varphi$. 
\end{itemize}
\end{proposition}
\begin{proof}
\begin{itemize}
\item[]
\item[(a)] results immediately from the definition of lower semicontinuity.
\item[(b)] Assume that $\operatorname*{dom}\varphi$ is open and $\varphi$ is lower semicontinuous on $\operatorname*{dom}\varphi$. 
Propose that there exists some $t_0 \in \mathbb{R}$ for which $\operatorname*{lev}_{\varphi,>}(t_0 )$ is not open.
$\Rightarrow\exists x^0 \in \operatorname*{lev}_{\varphi,>}(t_0 )\;\forall U\in\mathcal{N}(x^0 ):\; U\not\subseteq \operatorname*{lev}_{\varphi,>}(t_0 )$.
Since $\varphi$ is lower semicontinuous on $\operatorname*{dom}\varphi$, there exists some neighborhood V of $x^0$ with $V\subseteq \operatorname*{lev}_{\varphi,>}(t_0 )\cup(X\setminus \operatorname*{dom}\varphi )$.
Consider some arbitrary neighborhood $V_0$ of $x^0$.
$V_1:=V_0\cap V\subseteq \operatorname*{lev}_{\varphi,>}(t_0 )\cup(X\setminus \operatorname*{dom}\varphi)$, but $V_1\not\subseteq \operatorname*{lev}_{\varphi,>}(t_0 ) $.
$\Rightarrow \exists v_1\in V_1:\; v_1\in X\setminus \operatorname*{dom}\varphi.
\Rightarrow \forall V_0\in\mathcal{N}(x^0 )\;\exists v_1\in V_0:\; v_1\in X\setminus \operatorname*{dom}\varphi.
\Rightarrow x^0\in \operatorname*{cl}(X\setminus \operatorname*{dom}\varphi)=X\setminus \operatorname*{dom}\varphi$ since $\operatorname*{dom}\varphi$ is open. This contradicts $x^0 \in \operatorname*{lev}_{\varphi,>}(t_0 )$ and yields the assertion.
\item[(c)]Consider some lower bound $c$ of $\varphi$.
If $\operatorname*{lev}_{\varphi,>}(c-1)$ is open, 
$ A:=X\setminus \operatorname*{lev}_{\varphi,>}(c-1)=\{x\in X \mid \varphi(x) \leq c-1\}\cup(X\setminus \operatorname*{dom}\varphi)=X\setminus \operatorname*{dom}\varphi$ has to be closed
and thus $\operatorname*{dom}\varphi$ is open. This implies (c) because of (b).
\end{itemize}
\end{proof}

Example \ref{ex-extended-func} shows a continuous functional $\varphi$ for which not all sets $\operatorname*{lev}_{\varphi,>}(t)$ are open.

Replacing the function value $\varphi (\frac{\pi}{2})$ in Example \ref{ex-extended-func} by $\nu$, we get a functional $\varphi$ which is continuous on $\operatorname*{dom}\varphi=\{x\in\mathbb{R}
\mid -\frac{\pi}{2}\le x < \frac{\pi}{2}\}$ and for which all sets $\operatorname*{lev}_{\varphi,>}(t)$, $t\in\mathbb{R}$, are open, though $\operatorname*{dom}\varphi$ is not open.

\begin{proposition}\label{prop-semicon-levels-closed}
Let  $\varphi: X\rightarrow \overline{\mathbb{R}}_{\nu }$ be a nontrivial function on a topological space $X$.
\begin{itemize}
\item[(a)]  $\varphi$ is lower semicontinuous on $\operatorname*{dom}\varphi$ if the sublevel sets $\operatorname*{lev}_{\varphi,\le}(t)$ 
are closed for all $t\in \mathbb{R}$. 
\item[(b)] Suppose that $\operatorname*{dom}\varphi$ is closed. Then \\
$\varphi$ is lower semicontinuous if and only if the sublevel sets $\operatorname*{lev}_{\varphi,\le }(t)$ are closed for all $t\in \mathbb{R}$. 
\item[(c)] If $\varphi$ is bounded above, then:\\
The sublevel sets $\operatorname*{lev}_{\varphi,\le}(t)$ 
are closed for all $t\in \mathbb{R}$ if and only if $\operatorname*{dom}\varphi$ is closed and $\varphi$ is lower semicontinuous on $\operatorname*{dom}\varphi$. 
\end{itemize}
\end{proposition}
\begin{proof}
\begin{itemize}
\item[]
\item[(a)] $\operatorname*{lev}_{\varphi,\le}(t)$ is closed if and only if $A(t):=X\setminus \operatorname*{lev}_{\varphi,\le}(t)=\operatorname*{lev}_{\varphi,>}(t)\cup(X\setminus \operatorname*{dom}\varphi )$ is open. If $A(t)$ is open for all $t\in\mathbb{R}$, the lower semicontinuity of the functional on $\operatorname*{dom}\varphi$ follows from the definition of this property.
\item[(b)] The definition of lower semicontinuity of $\varphi$ implies that the set $A(t)$ is open for each $t\in\mathbb{R}$ for which $\operatorname*{lev}_{\varphi,>}(t)$ is not empty.
Thus the set $A(t)$ is open for each $t\in\mathbb{R}$ if $\operatorname*{dom}\varphi$ is closed. This yields assertion (b).
\item[(c)]Consider some upper bound $b$ of $\varphi$.
If $\operatorname*{lev}_{\varphi,\le}(b)$ is closed, 
$ A:=X\setminus \operatorname*{lev}_{\varphi,\le}(b)=\{x\in X \mid \varphi(x) > b\}\cup(X\setminus \operatorname*{dom}\varphi)=X\setminus \operatorname*{dom}\varphi$ has to be open and
thus $\operatorname*{dom}\varphi$ is closed. This implies (c) because of (b).
\end{itemize}
\end{proof}

A function can be lower semicontinuous on its domain though not all sublevel sets $\operatorname*{lev}_{\varphi,\le }(t)$ are closed.
\begin{example}\label{notlsc}
Consider the functional $\varphi : \mathbb{R}\to \overline{\mathbb{R}}_{\nu}$ given by
\[ \varphi(x)=\left\{
\begin{array}{r@{\quad\mbox{ if }\quad}l}
x & x>0,\\
\nu & x\leq 0.
\end{array}
\right.
\]
$\varphi$ is continuous on $\operatorname*{dom}\varphi$, but $\operatorname*{lev}_{\varphi,\le }(t)$ is not closed if $t>0$.
\end{example}

In Example \ref{ex-betrag}, the sublevel sets $\operatorname*{lev}_{\varphi,\le}(t)$ are closed for all $t\in\mathbb{R}$, though $\varphi$ is continuous on $\operatorname*{dom}\varphi$ and
$\operatorname*{dom}\varphi$ is not closed. 

For upper semicontinuity, analogous statements as in the previous two propositions follow with reverse relations and reverse directions of boundedness from Proposition \ref{p-usc-lsc}.

Let us now investigate the connection between lower semicontinuity and closedness of the epigraph.
\begin{lemma}\label{l-epi-level}
Let $X$ be a topological space, $\varphi: X\rightarrow \overline{\mathbb{R}}_{\nu }$.\\
$ \operatorname*{epi}\varphi$ is closed in $X \times \mathbb{R}$ if and only if the sublevel sets $\operatorname*{lev}_{\varphi,\le}(t)= \{x\in X \mid \varphi (x)\le t\}$ 
are closed for all $t\in \mathbb{R}$. 
\end{lemma}

\begin{proof}
\begin{itemize}
\item[]
\item[(i)] $ \operatorname*{epi}\varphi=\{(x,t)\in X \times \mathbb{R} \mid \varphi(x)\le t\}$ is closed iff
$ (X\times \mathbb{R} )\setminus \operatorname*{epi}\varphi=\{(x,t)\in X \times \mathbb{R} \mid \varphi(x) > t\}\cup((X\setminus \operatorname*{dom}\varphi) \times \mathbb{R})$ is open.
Then for each $t\in \mathbb{R}$, $\{x\in X \mid \varphi(x) > t\}\cup(X\setminus \operatorname*{dom}\varphi)$ is open and thus $\{x\in X \mid \varphi(x) \le t\}$ is closed.
\item[(ii)] Suppose that $\operatorname*{lev}_{\varphi,\le}(t)$ is closed for each $t\in \mathbb{R}$.
Assume that $ \operatorname*{epi}\varphi$ is not closed. Then $ (X\times \mathbb{R} )\setminus \operatorname*{epi}\varphi$ is not open.
$\Rightarrow \exists (x,\lambda )\in (X\times \mathbb{R} )\setminus \operatorname*{epi}\varphi\;\forall V\in\mathcal{N}((x,\lambda ))\;\exists (x^{V},t_{V})\in V:\; 
(x^{V},t_{V})\in\operatorname*{epi}\varphi$ , i.e. $\varphi (x^{V})\leq t_{V}$.
Consider an arbitrary $\epsilon >0$, $U_{\epsilon }(\lambda ):=\{r\in \mathbb{R}\mid \lambda - \epsilon < r < \lambda + \epsilon\}$ and an arbitrary $U\in\mathcal{N}(x)$.
$\Rightarrow V_{\epsilon }:=U\times U_{\epsilon }(\lambda )\in\mathcal{N}((x,\lambda )).
\Rightarrow \exists (x^{\epsilon },t_{\epsilon })\in V_{\epsilon }:\; 
\varphi (x^{\epsilon })\leq t_{\epsilon } < \lambda +\epsilon . 
\Rightarrow \forall U\in\mathcal{N}(x)\;\exists x^{\epsilon } \in U:\; x^{\epsilon } \in \operatorname*{lev}_{\varphi,\le}(\lambda +\epsilon ).
\Rightarrow x\in\operatorname*{cl}(\operatorname*{lev}_{\varphi,\le}(\lambda +\epsilon ))=\operatorname*{lev}_{\varphi,\le}(\lambda +\epsilon ).
\Rightarrow \varphi (x)\leq\lambda +\epsilon \;\forall\epsilon >0.
\Rightarrow \varphi (x)\leq\lambda$, a contradiction to $(x,\lambda )\notin \operatorname*{epi}\varphi$. Thus $ \operatorname*{epi}\varphi$ is closed.
\end{itemize}
\end{proof}

Corollary \ref{hypo-closed} results in an analogous statement for $ \operatorname*{hypo}\varphi$ and the superlevel sets $\operatorname*{lev}_{\varphi,\ge}(t)$.

Proposition \ref{prop-semicon-levels-closed} implies because of Lemma \ref{l-epi-level}:
\begin{proposition}\label{p-semi-epi}
Let  $\varphi: X\rightarrow \overline{\mathbb{R}}_{\nu }$ be a nontrivial function on a topological space $X$.
\begin{itemize}
\item[(a)] $\varphi$ is lower semicontinuous on $\operatorname*{dom}\varphi$ if
$ \operatorname*{epi}\varphi$ is closed in $X\times \mathbb{R}$. 
\item[(b)] If $\operatorname*{dom}\varphi$ is closed, then \\
$\varphi$ is lower semicontinuous $\iff$ $ \operatorname*{epi}\varphi$ is closed in $X\times \mathbb{R}$. 
\item[(c)] If $\varphi$ is bounded above, then:\\
$ \operatorname*{epi}\varphi$ is closed in $X\times \mathbb{R}$ if and only if $\operatorname*{dom}\varphi$ is closed and $\varphi$ is lower semicontinuous on $\operatorname*{dom}\varphi$. 
\end{itemize}
\end{proposition}

The function $\varphi$ in Example \ref{notlsc} is lower semicontinuous on its domain though $\operatorname*{epi}(\varphi)$ is not closed.
In Example \ref{ex-betrag}, 
$ \operatorname*{epi}\varphi$ is closed in $X\times \mathbb{R}$, though $\varphi$ is continuous on $\operatorname*{dom}\varphi$ and $\operatorname*{dom}\varphi$ is not closed.

For upper semicontinuity, we can prove an analogous statement as in Proposition \ref{p-semi-epi} by replacing the epigraph by the hypograph and changing the direction of boundedness.

Corollary \ref{c-semi-induced} and Lemma \ref{l-epi-level} imply:
\begin{corollary}\label{lsc-classic}
Let $X$ be a topological space, $\varphi: X\rightarrow \overline{\mathbb{R}}$. \\
Then the following statements are equivalent:
\begin{itemize}
\item[(a)] $\varphi$ is lower semicontinuous.
\item[(b)] The sets $\operatorname*{lev}_{\varphi,>}(t)$ 
are open for all $t\in \mathbb{R}$.
\item[(c)] The sublevel sets $\operatorname*{lev}_{\varphi,\le}(t)$ 
are closed for all $t\in \mathbb{R}$. 
\item[(d)] $ \operatorname*{epi}\varphi$ is closed in $X\times \mathbb{R}$. 
\end{itemize}
\end{corollary}

\begin{remark}
The characterization of lower semicontinuity in Corollary \ref{lsc-classic} describes the ways in which lower semicontinuity is defined in the classical framework which works with $+\infty$ instead of $\nu$. Rockafellar \cite[Theorem 7.1]{Roc97} proved the equivalence between these conditions for $\varphi: \mathbb{R}^{n}\rightarrow \overline{\mathbb{R}}$.
Moreau \cite{Mor67} defined lower semicontinuity by property (c) and stated the equi\-valence with property (d) for $X$ being a topological space.
\end{remark}
\smallskip

\section{Convexity}\label{s-cx}

We will now study convexity of extended real-valued functions.
\begin{definition}\label{def-conv-allg}
Let $X$ be a linear space and $\varphi: X\rightarrow \overline{\mathbb{R}}_{\nu }$. \\
$\varphi$ is said to be \textbf{convex} if $\operatorname*{dom}\varphi$ and $\operatorname*{epi}\varphi$ are convex sets. \\
$\varphi$ is \textbf{concave} if  $\operatorname*{dom}\varphi$ and $\operatorname*{hypo}\varphi$ are convex sets. 
\end{definition}

\begin{remark}
The characterization of a real-valued convex function on some finite-dimensional vector space by the epigraph of the function was given by Fenchel \cite[p. 57]{Fen53}.
\end{remark}

If $\varphi: X \to \overline{\mathbb{R}}_{\nu } $ does not attain the value $+\infty$, then convexity of $\operatorname*{epi}\varphi$ implies convexity of $\operatorname*{dom}\varphi$.

A continuous functional with a convex epigraph and a convex hypograph does not necessarily have a convex domain.
\begin{example}
Define $\varphi : \mathbb{R}\to \overline{\mathbb{R}}_{\nu}$ by
\[ \varphi(x)=\left\{
\begin{array}{r@{\quad\mbox{ if }\quad}l}
-\infty & x=-1,\\
+\infty & x=1,  \\
\nu & x\in \mathbb{R}\setminus \{ -1,1\} .
\end{array}
\right.
\] 
Then $\operatorname*{dom}\varphi =\{ -1,1\}$ is closed, but not convex. $\varphi$ is continuous.
$\operatorname*{epi}\varphi =\{ -1\}\times \mathbb{R}$ and $\operatorname*{hypo}\varphi =\{ 1\}\times \mathbb{R}$ are closed convex sets.
\end{example}

Lemma \ref{l-epihypo} implies:
\begin{lemma}\label{l-cx-cv}
Let $X$ be a linear space and $\varphi: X\rightarrow \overline{\mathbb{R}}_{\nu }$. \\
Then $\varphi$ is concave on $X$ if and only if $-\varphi$ is convex on $X$.
\end{lemma}

Definition \ref{def-conv-allg} is compatible with the usual definition of convex real-valued functions.

\begin{theorem}\label{tconv}
Let $X$ be a linear space and $\varphi: X \to \overline{\mathbb{R}}_{\nu } $ be proper. \\
Then $\varphi$ is convex or concave if and only if  $\operatorname*{dom} \varphi$ is convex and
\begin{eqnarray}
\varphi(\lambda x^1+(1-\lambda)x^2)\leq \lambda \varphi(x^1)+(1-\lambda) \varphi(x^2) & \mbox{ or }\label{eq-cx} \\
\varphi(\lambda x^1+(1-\lambda)x^2)\geq \lambda \varphi(x^1)+(1-\lambda) \varphi(x^2), & \mbox{respectively},\nonumber
\end{eqnarray}
holds for all $x^1, x^2\in \operatorname*{dom} \varphi$, $\lambda\in(0,1)$.
\end{theorem}
\begin{proof} 
\begin{itemize}
\item[]
\item[(a)] If $\varphi$ is convex, then $\operatorname*{epi}  \; \varphi$ is a convex set. Consider arbitrary elements $x^1, x^2\in \operatorname*{dom} \varphi$, $\lambda\in(0,1)$.
Because of $(x^1,\varphi (x^1)), (x^2,\varphi (x^2))\in\operatorname*{epi}\varphi$, we get: $\lambda \cdot (x^1,\varphi (x^1))+(1-\lambda) \cdot (x^2,\varphi (x^2))\in \operatorname*{epi}\varphi$.
Hence $\varphi(\lambda x^1+(1-\lambda)x^2)\leq \lambda \varphi(x^1)+(1-\lambda) \varphi(x^2)$.
\item[(b)] Assume that $\operatorname*{dom} \varphi$ is convex and
$ \varphi(\lambda
 x^1+(1-\lambda)x^2)\leq \lambda \varphi(x^1)+(1-\lambda) \varphi(x^2)$
for all $x^1, x^2\in \operatorname*{dom} \varphi$, $\lambda\in(0,1)$.
Consider $(x^1,t_1), (x^2,t_2)\in\operatorname*{epi}\varphi$, i.e. $x^1, x^2\in\operatorname*{dom} \varphi$ with $\varphi(x^1)\le t_1, \varphi(x^2)\le t_2$.
Then $\lambda  x^1+(1-\lambda)x^2\in\operatorname*{dom} \varphi$ and $ \varphi(\lambda
 x^1+(1-\lambda)x^2)\le \lambda \varphi(x^1)+(1-\lambda) \varphi(x^2)\le \lambda t_1+(1-\lambda) t_2$
for all $\lambda\in(0,1)$. Hence $\lambda \cdot (x^1,t_1)+(1-\lambda) \cdot (x^2,t_2)\in \operatorname*{epi}  \; \varphi$ for all $\lambda\in(0,1)$, i.e. 
$\operatorname*{epi}  \; \varphi$ is convex.
\end{itemize}
Thus the assertion holds for convex functionals. This, together with Lemma \ref{l-cx-cv}, implies the statement for the concave functional.
\end{proof}

Let us now characterize improper convex functionals. \\
A continuous convex functional can attain the value $+\infty$.

\begin{example}
Let $\varphi : \mathbb{R}\to \overline{\mathbb{R}}_{\nu}$ be given by
\[ \varphi(x)=\left\{
\begin{array}{r@{\quad\mbox{ if }\quad}l}
\nu & x <0,\\
+\infty & x=0,\\
\frac{1}{x} & x>0.
\end{array}
\right.
\]
$\varphi$ is a continuous convex functional.
\end{example}

\begin{proposition}\label{p-conv-infty}
Let  $\varphi: X\rightarrow \overline{\mathbb{R}}_{\nu }$ be a convex function on a linear space $X$.
\begin{itemize}
\item[(a)] $\operatorname*{dom}_{-}\varphi :=\{x\in \operatorname*{dom}\varphi\mid\varphi (x)\not= +\infty\}$ is convex.
\item[(b)] If $\varphi (x^0)= -\infty$ for some $x^0\in X$, then $\varphi (\lambda x^0+(1-\lambda) x)=-\infty$ for all $x\in \operatorname*{dom}_{-}\varphi , \lambda\in (0,1)$. 
\item[(c)]\label{p-conv-proper}
If there exists some $x^0 \in X$ such that $\varphi (x^0) = - \infty $, then $\varphi(x) = - \infty$ holds for all $x\in\operatorname*{icr}(\operatorname*{dom}_{-}\varphi)$.
\item[(d)]\label{cx_odd_infty}
If $\{x,-x\}\subset\operatorname*{dom}\varphi$, $\varphi(x)=-\infty$ and $\varphi(0)\not=-\infty$, then $\varphi(-\lambda x)=+\infty$ for all $\lambda\in (0,1]$.
\item[(e)]\label{p-Rockproperf}
Assume that $X$ is a topological vector space, that $\varphi$ is lower semicontinuous on $\operatorname*{dom}_{-}\varphi$ and attains the value $-\infty$. Then $\varphi$ has no finite values.
\end{itemize}
\end{proposition}
\begin{proof}
\begin{itemize}
\item[]
\item[(a)] Consider $x^1, x^2\in \operatorname*{dom}_{-}\varphi$, $\lambda\in (0,1)$, $x:=\lambda x^1+(1- \lambda) x^2$.
There exist $t_1, t_2\in \mathbb{R}$ such that $(x^1,t_1),(x^2,t_2)\in \operatorname*{epi}\varphi$.
Since $\operatorname*{epi}  \; \varphi$ is convex, we get for $x$ and $t:=\lambda t_1+(1- \lambda) t_2\in \mathbb{R}$ that $(x,t)\in\operatorname*{epi}\varphi$.
Hence $x\in \operatorname*{dom}\varphi$ and $\varphi (x)\le t$, which implies $x\in \operatorname*{dom}_{-}\varphi$.
\item[(b)] Assume $x^0 \in X$ with $\varphi (x^0) = - \infty $ and $x^1\in \operatorname*{dom}_{-}\varphi$.
$\Rightarrow \forall \lambda\in(0,1):\quad x^{\lambda}:=\lambda x^0+(1-\lambda )x^1\in \operatorname*{dom}_{-}\varphi$ because of (a).
$\exists t_1 \in\mathbb{R} :\; (x^1 ,t_1 )\in \operatorname*{epi}\varphi$.\\
Consider an arbitrary $\lambda\in(0,1)$, $t_\lambda\in\mathbb{R}$. 
$t:=\frac{1}{\lambda}t_\lambda -\frac{1-\lambda}{\lambda}t_1 \in \mathbb{R}$.
$\Rightarrow t_\lambda=\lambda t+(1-\lambda )t_1$.
$\Rightarrow (x^\lambda,t_\lambda)\in\operatorname*{epi}\varphi$, because $(x^0,t), (x^1,t_1)\in\operatorname*{epi}\varphi$ and $\operatorname*{epi}\varphi$ is convex.
Since $(x^\lambda,t_\lambda)\in\operatorname*{epi}\varphi$ for each $t_\lambda\in \mathbb{R}$, we get  
$\{x^\lambda \}\times\mathbb{R}\subseteq\operatorname*{epi}\varphi$ and $\varphi(x^\lambda)=-\infty$ for each $\lambda\in(0,1)$.
\item[(c)] Assume $x^0 \in X$ with $\varphi (x^0) = - \infty $ and $x^1\in \operatorname*{icr}(\operatorname*{dom}_{-}\varphi)$.
$\Rightarrow \exists\epsilon >0:\; x^2:=x^1+\epsilon (x^1-x^0)\in \operatorname*{dom}_{-}\varphi$.
$x^1=\lambda x^2+(1-\lambda) x^0$ with $\lambda:=\frac{1}{1+\epsilon}\in (0,1)$.
$\varphi(x^1)=-\infty$ follows from (b).
\item[(d)] Assume, under the assumptions of (d), that there exists some $\lambda\in (0,1]$ with $\varphi(-\lambda x)\not=+\infty$. $\Rightarrow 
-\lambda x\in \operatorname*{dom}_{-}\varphi$. $\Rightarrow \varphi(0)=-\infty$ because of (b), a contradiction.
\item[(e)] Assume $x^0 \in X$ with $\varphi (x^0) = - \infty $ and $x\in \operatorname*{dom}_{-}\varphi$.
$\Rightarrow \forall \lambda\in(0,1):\quad x^{\lambda}:=\lambda x^0+(1-\lambda )x\in \operatorname*{dom}_{-}\varphi$ and $\varphi(x^\lambda)=-\infty$ because of (b).\\
$\forall U\in\mathcal{N}(x)\;\exists\lambda\in(0,1):\; x^{\lambda}\in U$.
This results in $\varphi(x)=-\infty$ because of the definition of lower semicontinuity.
\end{itemize}
\end{proof}

Let us mention that $\operatorname*{dom}_{-}\varphi$ is the projection of $\operatorname*{epi}\varphi$ onto $X$.
\begin{remark}
Part (c) of the proposition can be found in \cite{Zal:02}, part (c) and (e) for $X=\mathbb{R}^n$ in \cite{Roc97}, but both authors define convexity with inequality (\ref{eq-cx}) and refer to the classical framework, where $+\infty$ is used instead of $\nu$ and points with function value $+\infty$ are excluded from the effective domain.
\end{remark}

\begin{lemma}\label{l-cx-anders}
Let $X$ be a linear space and $\varphi: X\rightarrow \overline{\mathbb{R}}_{\nu }$.\\
If $\varphi$ is convex or concave, then each of the following functions has the same property:
\begin{itemize}
\item[(a)] $\lambda\varphi$ with $\lambda\in\mathbb{R}_{>}$,
\item[(b)] $\varphi +c$ with $c\in\mathbb{R}$,
\item[(c)] $g:X \rightarrow \overline{\mathbb{R}}_{\nu}$ defined by $g(x)=\varphi (x-\overline{x})$ with $\overline{x}\in X$,
\item[(d)] $g:X \rightarrow \overline{\mathbb{R}}_{\nu}$ defined by $g(x)=\varphi (\lambda x)$ with $\lambda\in \mathbb{R}\setminus\{0\}$.
\end{itemize}
\end{lemma}
\smallskip

\section{Linearity and Related Algebraic Properties}\label{sec-lin-ext}

Usually, a linear function is a function that maps into a vector space. This is not the case for an extended real-valued function.
\begin{definition}\label{def-lin-fktal}
Let $X$ be a linear space and $\varphi: X\rightarrow \overline{\mathbb{R}}_{\nu }$.\\
$\varphi$ is an \textbf{affine} functional if  $\varphi$ is convex and concave.\\
If $\operatorname*{dom}\varphi =X$, then $\varphi$ is a \textbf{linear} functional if  $\varphi$ is convex and concave and $\varphi (0)=0$ holds. 
\end{definition}

\begin{remark}
Fenchel \cite[p. 59]{Fen53} proved that real-valued affine functionals on finite-dimensional vector spaces defined in the traditional way by an equality are just the functionals which are convex and concave.
\end{remark}

Obviously, each functional with a nonempty, convex effective domain which is constant on this domain with the value $c$ is affine and, in the case $c=0$, also linear.

We get from Definition \ref{def-lin-fktal}:
\begin{lemma}
Let $X$ be a linear space and $\varphi: X \to \overline{\mathbb{R}}$. \\
$\varphi$ is an affine functional with $\varphi (0)\in\mathbb{R}$ if and only if  it is the sum of some linear functional $\varphi_l: X \to \overline{\mathbb{R}}$ and 
some real value.
\end{lemma}

An affine functional is not necessarily the sum of a linear functional and a constant value.
\begin{example}
Define $\varphi : \mathbb{R}\to \overline{\mathbb{R}}$ by
\[ \varphi(x)=\left\{
\begin{array}{r@{\quad\mbox{ if }\quad}l}
-\infty & x\leq 0,\\
+\infty & x>0 .
\end{array}
\right.
\]
$\operatorname*{epi}\varphi =(-\infty,0]\times \mathbb{R}$ is convex and closed. 
$\operatorname*{hypo}\varphi = (0,+\infty)\times \mathbb{R}$ is a convex set.
$\varphi$ is an affine, lower semicontinuous functional, but not the sum of a linear functional and some value $c\in\overline{\mathbb{R}}$.
\end{example}

We get from Lemma \ref{l-cx-anders}:
\begin{lemma}\label{l-aff-anders}
Let $X$ be a linear space and $\varphi: X\rightarrow \overline{\mathbb{R}}_{\nu }$.\\
If $\varphi$ is affine, then each of the following functions is also affine:
\begin{itemize}
\item[(a)] $\lambda\varphi$ with $\lambda\in\mathbb{R}$,
\item[(b)] $\varphi +c$ with $c\in\mathbb{R}$,
\item[(c)] $g:X \rightarrow \overline{\mathbb{R}}_{\nu }$ defined by $g(x)=\varphi (x-\overline{x})$ with $\overline{x}\in X$.
\item[(d)] $g:X \rightarrow \overline{\mathbb{R}}_{\nu }$ defined by $g(x)=\varphi (\lambda x)$ with $\lambda\in \mathbb{R}$.
\end{itemize}
If $\operatorname*{dom}\varphi =X$ and $\varphi$ is linear, then the functions in {\rm (a)} and {\rm (d)} are also linear. 
\end{lemma}

A functional which is affine and lower semicontinuous on its domain is not necessarily continuous.

\begin{example}\label{ex-cv-cx1}
Define $\varphi : \mathbb{R}\to \overline{\mathbb{R}}_{\nu}$ by
\[ \varphi(x)=\left\{
\begin{array}{r@{\quad\mbox{ if }\quad}l}
1 & x=0,\\
+\infty & x\in (0,1],  \\
\nu & x\in \mathbb{R}\setminus [0,1] .
\end{array}
\right.
\]
$\operatorname*{dom}\varphi =[0,1]$ is closed and convex. $\operatorname*{epi}\varphi =\{ 0\}\times [1,+\infty)$ is convex and closed. $\operatorname*{hypo}\varphi = (\{ 0\}\times (-\infty, 1])\cup ((0,1]\times \mathbb{R})$ is convex, but not closed.
$\varphi$ is an improper convex and concave functional which is lower semicontinuous on its domain, but not continuous.
\end{example}

Let us now prove a statement (cp. \cite{Wei90}) which we will use in the next proposition.
\begin{lemma}\label{l-bd-generate}
Assume that $(X,\tau )$ is a topological vector space, $A\subset X$, $a\in A$, $x\in X\setminus \operatorname*{cl}A$.
\begin{itemize}
\item[(a)] There exists some $t\in (0,1]$ such that $ta+(1-t)x\in\operatorname*{bd}A$.
\item[(b)] There exists some $t\in (0,1]$ such that $ta+(1-t)x\in\operatorname*{bd}_M A$, where $\operatorname*{bd}_M A$ denotes the boundary of A w.r.t. the topology $\tau_M$ induced by $\tau$ on $M:=\{ta+(1-t)x\mid t\in [0,1]\}$.
\end{itemize}
\end{lemma}
\begin{proof}
\begin{itemize}
\item[]
\item[(a)] $t_0:=\operatorname*{inf}\{t\in\mathbb{R}_+ :\;ta + (1-t)x\in A\}=\operatorname*{inf}\{t\in\mathbb{R}_+ :\;x + t(a-x)\in A\}\leq 1$ since $a\in A$.\\
Assume $t_0=0$. $\Rightarrow \;\forall U\in\mathcal{N}(x)\;\exists t\in\mathbb{R}_+:\;
x+t(a-x)\in A\cap U$.
$\Rightarrow x\in\operatorname*{cl}A$, a contradiction. Thus $t_0>0$.\\
$b:=t_0 a+(1-t_0)x$. The definition of $t_0$ implies:
$\forall V\in\mathcal{N}(b)\;\exists t_1\geq t_0\;\exists t_2<t_0:\;
t_1,t_2\in [0,1], x + t_1 (a-x)\in A\cap V$ and $x + t_2 (a-x)\in V\setminus A$.
Hence $b\in\operatorname*{bd}A$.
\item[(b)] can be proved in the same way as (a) when replacing the neighborhoods by their intersection with M.
\end{itemize}
\end{proof}

\begin{proposition}\label{p-const}
Let  $(X,\tau)$ be a topological vector space and $\varphi: X\rightarrow \overline{\mathbb{R}}_{\nu }$. 
If $\varphi$ is affine and continuous on its effective domain, then it is constant on $\operatorname*{dom}\varphi$ or proper.
\end{proposition}
\begin{proof}
We consider a functional $\varphi : X\to \overline{\mathbb{R}}_{\nu}$ which is convex, concave and continuous on its effective domain, but not proper.
Hence part (e) of Proposition \ref{p-conv-infty} implies that $\varphi$ does not attain any real value.\\
Assume that $\varphi$ is not constant on $\operatorname*{dom}\varphi$.
Then there exist $x^1, x^2\in \operatorname*{dom}\varphi$ with $\varphi(x^1)=-\infty, \varphi(x^2)=+\infty$.
$M:=\{tx^1+(1-t)x^2\mid t\in[0,1]\}\subseteq \operatorname*{dom}\varphi$ since $\operatorname*{dom}\varphi$ is convex. $A:=M\cap\operatorname*{dom}_{-}\varphi$.\\
Since $\varphi$ is continuous and $\operatorname*{dom}_{-}\varphi =\{x\in \operatorname*{dom}\varphi\mid\varphi (x) = -\infty\}$, there exists some $U\in\mathcal{N}(x^2)$ 
with $U\cap \operatorname*{dom}_{-}\varphi =\emptyset$. Thus $x^2\notin \operatorname*{cl}\operatorname*{dom}_{-}\varphi$.
Because of $x^1\in A$ and $x^2\notin \operatorname*{cl}A$, Lemma \ref{l-bd-generate} implies the existence of some $t\in (0,1]$ such that $a:=tx^1+(1-t)x^2\in\operatorname*{bd}_M A$, where 
$\operatorname*{bd}_M A$ denotes the boundary of $A$ w.r.t. the topology induced by $\tau$ on $M$.
$a\in\operatorname*{dom}\varphi$ since $\operatorname*{dom}\varphi$ is convex. The continuity of $\varphi$ implies $\varphi (a)=-\infty$ by the definition of $A$ and the existence of some neighborhood $V$ of $a$ such that $\varphi(x)=-\infty$ for all $x\in V\cap\operatorname*{dom}\varphi$. 
Since $a\in\operatorname*{bd}_M A$, there exists some $x\in V\cap M:\;x\notin A$. This implies $x\notin \operatorname*{dom}_{-}\varphi$, a contradiction.
Consequently, $\varphi$ is constant on $\operatorname*{dom}\varphi$.
\end{proof}

\begin{corollary}\label{cont-lin-prop}
Assume that $X$ is a topological vector space and that $\varphi: X \to \overline{\mathbb{R}}$ is continuous and linear.
Then $\varphi$ is finite-valued.
\end{corollary}

A linear functional may be improper without being constant.
\begin{example}\label{bsp-lin-improp}
Define $\varphi : \mathbb{R}\to \overline{\mathbb{R}}$ by
\[ \varphi(x)=\left\{
\begin{array}{r@{\quad\mbox{ if }\quad}l}
-\infty & x<0,\\
0 & x=0,  \\
+\infty & x>0 .
\end{array}
\right.
\]
$\operatorname*{epi}\varphi =((-\infty,0)\times \mathbb{R})\cup (\{ 0\}\times [0,+\infty))$ and 
$\operatorname*{hypo}\varphi = ((0,+\infty)\times \mathbb{R})\cup (\{ 0\}\times (-\infty, 0])$ are convex, but not closed.
$\varphi$ is an improper linear functional.
\end{example}

Definition \ref{def-lin-fktal} is compatible with the usual definition of linear and affine real-valued functions. This will be shown in Theorem \ref{t-lin}.

Let us first introduce further algebraic properties of extended real-valued functions.

\begin{definition}\label{def-sub-fktal}
Let $X$  be a linear space and $\varphi: X\rightarrow \overline{\mathbb{R}}_{\nu }$. \\
$\varphi$ is said to be
\begin{itemize}
\item[(a)] \textbf{positively homogeneous} if  $\operatorname*{dom}\varphi$ and $\operatorname*{epi}\varphi$ are nonempty cones,
\item[(b)] \textbf{subadditive} if  $\operatorname*{dom}\varphi +\operatorname*{dom}\varphi \subseteq \operatorname*{dom}\varphi$ and $\operatorname*{epi}\varphi +\operatorname*{epi}\varphi \subseteq \operatorname*{epi}\varphi$ hold,
\item[(c)] \textbf{superadditive} if  $\operatorname*{dom}\varphi +\operatorname*{dom}\varphi \subseteq \operatorname*{dom}\varphi$ and $\operatorname*{hypo}\varphi +\operatorname*{hypo}\varphi \subseteq \operatorname*{hypo}\varphi$ hold,
\item[(d)] \textbf{additive} if it is subadditive and superadditive,
\item[(e)] \textbf{sublinear} if $\operatorname*{dom}\varphi$ and $\operatorname*{epi}\varphi$ are nonempty convex cones,
\item[(f)] \textbf{odd} if $\operatorname*{dom}\varphi = -\operatorname*{dom}\varphi$ and $\varphi (-x)=-\varphi (x)$ is satisfied for all 
$x\in \operatorname*{dom}\varphi$,
\item[(g)] \textbf{homogeneous} if $\varphi$ is positively homogeneous and odd,
\end{itemize}
\end{definition}

According to our definition, $\varphi (0)\in\{0,-\infty\}$ holds for each positively homogeneous and for each sublinear functional.
$\varphi$ is superadditive if and only if $-\varphi$ is subadditive.

\begin{lemma}\label{l-cxCone}
Let $C$ be a nonempty subset of a linear space $X$.
\[
\begin{array}{l@{\quad\iff\quad}l}
C \mbox{ is a convex cone} & \lambda _1 c^1 + \lambda _2 c^2 \in C \mbox{ for all }\lambda _1,\lambda _2\in\mathbb{R}_+, c^1, c^2\in C,\\
 & C \mbox{ is a cone and } C+C\subseteq C,\\
 & C \mbox{ is convex, } C+C\subseteq C \mbox{ and } 0\in C. 
\end{array}
\]
\end{lemma}

This implies:

\begin{lemma}\label{l-sublin}
Let $X$  be a linear space and $\varphi: X\rightarrow \overline{\mathbb{R}}_{\nu }$. 
\[
\begin{array}{l@{\quad\iff\quad}l}
\varphi \mbox{ is sublinear } & \varphi \mbox{ is convex and positively homogeneous},\\
 & \varphi \mbox{ is subadditive and positively homogeneous},\\
 & \varphi \mbox{ is convex and subadditive and } \varphi (0)\leq 0. 
\end{array}
\]
\end{lemma}

The properties defined in Definition \ref{def-sub-fktal} coincide with the usual definitions for real-valued functions. 

\begin{proposition}\label{p-hom}
Let $X$  be a linear space and $\varphi: X\rightarrow \overline{\mathbb{R}}_{\nu }$. 
\begin{itemize}
\item[(a)]
$\varphi$ is positively homogeneous with $\varphi (0)=0$ if and only if  
$\operatorname*{dom}\varphi$ is a nonempty cone and \[\varphi (\lambda x)=\lambda \varphi (x)\] is satisfied for all $\lambda
\in \mathbb{R}_{+}$ and $x\in \operatorname*{dom}\varphi$.
\item[(b)]
A proper function $\varphi$ is subadditive or superadditive if and only if\\
$\operatorname*{dom}\varphi + \operatorname*{dom}\varphi \subseteq \operatorname*{dom}\varphi$ and 
\[ 
\begin{array}{l@{\quad }l}
\varphi (x^1 +  x^2) \leq \varphi (x^1) +  \varphi (x^2) & \mbox{ or }\\
\varphi (x^1 +  x^2) \geq \varphi (x^1) +  \varphi (x^2), & \mbox{respectively},
\end{array}
\]
holds for all $x^1, x^2 \in \operatorname*{dom}\varphi$.
\item[(c)] $\varphi$ is homogeneous if and only if 
$\operatorname*{dom}\varphi$ is a nonempty cone,\linebreak
$\operatorname*{dom}\varphi=-\operatorname*{dom}\varphi$ and \[\varphi (\lambda x)=\lambda \varphi (x)\] 
holds for all $\lambda \in \mathbb{R}$ and $x\in \operatorname*{dom}\varphi$.
\end{itemize}
\end{proposition}
\begin{proof} 
\begin{itemize}
\item[]
\item[(a)](i) Suppose first that $\varphi$ is positively homogeneous, i.e. that $\operatorname*{dom}\varphi$ and $\operatorname*{epi}\varphi$ are cones. Assume $\lambda \in \mathbb{R}_{+}\setminus\{ 0\}$.
Consider first some $x\in \operatorname*{dom}\varphi$ with $t:=\varphi (x)\in\mathbb{R}$. $\lambda\cdot (x,t)\in \operatorname*{epi}\varphi$ since $(x,t)\in \operatorname*{epi}\varphi$.
$\Rightarrow \varphi(\lambda x)\le \lambda t$.
Suppose $\varphi(\lambda x) < \lambda t$. $\Rightarrow \exists \lambda_1 \in \mathbb{R}$ with $ \lambda_1 < \lambda t:\quad \varphi(\lambda x)<\lambda_1$. 
$\Rightarrow (\lambda x,\lambda_1)\in \operatorname*{epi}\varphi$.
If $\lambda >0$, then $\frac{1}{\lambda}\cdot (\lambda x,\lambda_1)\in \operatorname*{epi}\varphi$, i.e. $(x,\frac{\lambda_1}{\lambda})\in \operatorname*{epi}\varphi$.
$\Rightarrow t=\varphi(x) \le \frac{\lambda_1}{\lambda}$, which implies $\lambda t \le \lambda_1$, a contradiction. Thus $\varphi(\lambda x) = \lambda t=\lambda \varphi(x)$. \\
Consider now some $x\in \operatorname*{dom}\varphi$ with $\varphi (x)=-\infty$. $\Rightarrow (x,t)\in \operatorname*{epi}\varphi$ for all $t\in\mathbb{R}$.
$\Rightarrow (\lambda x,\lambda t)\in \operatorname*{epi}\varphi$ for all $t\in\mathbb{R}$ since $\operatorname*{epi}\varphi$ is a cone.
$\Rightarrow \varphi (\lambda x)=-\infty=\lambda \varphi(x)$.\\
Consider finally some $x\in \operatorname*{dom}\varphi$ with $\varphi (x)=+\infty$. If $\varphi (\lambda x)\not= +\infty$, then $\varphi(x)=\varphi(\frac{1}{\lambda}\lambda x)=
\frac{1}{\lambda}\varphi (\lambda x)\not= +\infty$ by the above statements, a contradiction.  \\
(ii) Assume now that $\operatorname*{dom}\varphi$ is a cone and $\varphi (\lambda x)=\lambda \varphi (x)$ is satisfied for all $\lambda
\in \mathbb{R}_{+}$ and $x\in \operatorname*{dom}\varphi$. If $(x^1,t_1)\in\operatorname*{epi}\varphi$, then for each $\lambda_1\in \mathbb{R}_{+}$:\; $\lambda_1 x^1\in\operatorname*{dom}\varphi$ and $\varphi(\lambda_1 x^1)=\lambda_1\varphi(x^1)\leq \lambda_1 t_1$, hence $(\lambda_1 x^1,\lambda_1 t_1)\in\operatorname*{epi}\varphi$. 
Thus $\operatorname*{epi}\varphi$ is a cone.
\item[(b)] (i) Suppose first that $\varphi$ is subadditive, i.e. $\operatorname*{dom}\varphi + \operatorname*{dom}\varphi \subseteq \operatorname*{dom}\varphi$ and $\operatorname*{epi}\varphi +\operatorname*{epi}\varphi \subseteq \operatorname*{epi}\varphi$.
Consider $x^1, x^2 \in \operatorname*{dom}\varphi$.
$\Rightarrow (x^1,\varphi(x^1)),(x^2,\varphi(x^2))\in\operatorname*{epi}\varphi$.
$\Rightarrow (x^1,\varphi(x^1))+(x^2,\varphi(x^2))\in\operatorname*{epi}\varphi$.
$\Rightarrow \varphi(x^1+x^2)\leq\varphi(x^1)+\varphi(x^2)$.\\
(ii) Assume that $\operatorname*{dom}\varphi + \operatorname*{dom}\varphi \subseteq \operatorname*{dom}\varphi$ and 
that $\varphi (x^1 +  x^2) \leq \varphi (x^1) +  \varphi (x^2)$ holds for all $x^1, x^2 \in \operatorname*{dom}\varphi$.
If $(x^3,t_3),(x^4,t_4)\in \operatorname*{epi}\varphi$, then the assumption yields that $x^3+x^4\in \operatorname*{dom}\varphi$ and 
$\varphi(x^3+x^4)\leq \varphi (x^3) +  \varphi (x^4)\leq t_3+t_4$.
$\Rightarrow (x^3+x^4,t_3+t_4)\in\operatorname*{epi}\varphi$.\\
We have proved part (b) for subadditivity. The statement for a superadditive function $\varphi$ results from the subadditivity of $-\varphi$.
\item[(c)]
(i) Let us first assume that $\varphi$ is homogeneous, i.e. odd and positively homogeneous.
Then $\operatorname*{dom}\varphi$ is a cone, $\operatorname*{dom}\varphi=-\operatorname*{dom}\varphi$, $\varphi(-x)=-\varphi(x)$ and $\varphi (\lambda x)=\lambda \varphi (x)$ 
hold for all $\lambda \in \mathbb{R}_{+}$ and $x\in \operatorname*{dom}\varphi$.
Consider some arbitrary $x\in \operatorname*{dom}\varphi$ and $\lambda_1 <0$. 
Then $\varphi(\lambda_1 x)=\varphi((-\lambda_1)\cdot (-x))=(-\lambda_1)\varphi (-x)=(-\lambda_1)\cdot (-\varphi(x))
=\lambda_1 \varphi(x)$. Consequently, $\varphi (\lambda x)=\lambda \varphi (x)$ 
holds for all $\lambda \in \mathbb{R}$ and $x\in \operatorname*{dom}\varphi$.\\
(ii) The reverse direction of the equivalence is obvious because of (a).
\end{itemize}
\end{proof}

The statement of Proposition \ref{p-hom} (b) cannot be extended to functions which are not proper. This is illustrated by Example \ref{bsp-lin-improp}. There the function $\varphi$ is additive, but $\varphi(-1+1)=0\not=\nu=\varphi(-1)+\varphi(+1)$.

\begin{lemma}\label{l-homo-anders}
Let $X$  be a linear space and $\varphi: X\rightarrow \overline{\mathbb{R}}_{\nu }$. 
\begin{itemize}
\item[(a)] If $\varphi$ is positively homogeneous, additive, odd or homogeneous, then
$\lambda\varphi$ with $\lambda\in\mathbb{R}$ has the same property.
\item[(b)] If $\varphi$ is subadditive, superadditive or sublinear, then
$\lambda\varphi$ with $\lambda\in\mathbb{R}_{+}\setminus\{ 0\}$ has the same property.
\item[(c)] If $\varphi$ is subadditive, then $\varphi +c$ with $c\in\mathbb{R}$, $c>0$, is subadditive.
\item[(d)] If $\varphi$ is positively homogeneous, subadditive, superadditve, additive, sublinear, odd or homogeneous, then
$g:X \rightarrow \overline{\mathbb{R}}_{\nu}$ defined by $g(x)=\varphi (\lambda x)$ with $\lambda\in \mathbb{R}\setminus\{0\}$ has the same property.
\end{itemize}
\end{lemma}

We are now going to investigate the relationship between additive and linear functions.
\begin{lemma}\label{l-subadd}
Let  $\varphi: X\rightarrow \overline{\mathbb{R}}_{\nu }$ be a subadditive function on a linear space $X$. 
\begin{itemize}
\item[(a)] If $\varphi(0)=-\infty$, then $\varphi$ does not attain any real value.
\item[(b)] If $\{x,-x\}\subset\operatorname*{dom}\varphi$, $\varphi(x)=-\infty$ and $\varphi(0)\not=-\infty$, then $\varphi(-x)=+\infty$.
\item[(c)] If $\varphi(0)\in\mathbb{R}$, then $\varphi(0)\ge 0$ and $\varphi(nx)=-\infty$ for all $n\in\mathbb{N}\setminus\{ 0\}$ and all $x\in\operatorname*{dom}\varphi$ with $\varphi(x)=-\infty$.
\end{itemize}
\end{lemma}
\begin{proof}
\begin{itemize}
\item[]
\item[(a)] Assume $\varphi(0)=-\infty$ and $\lambda:=\varphi (x)\in\mathbb{R}$ for some $x\in\operatorname*{dom}\varphi$.
$\Rightarrow (0,t)\in\operatorname*{epi}\varphi$ for all $t\in\mathbb{R}$, $(x,\lambda)\in\operatorname*{epi}\varphi$.
$\Rightarrow (0,t)+(x,\lambda)=(x,t+\lambda)\in\operatorname*{epi}\varphi$ for all $t\in\mathbb{R}$. $\Rightarrow \varphi (x)=-\infty$, a contradiction.
\item[(b)] Assume that $\varphi(-x)\not=+\infty$. $\Rightarrow \exists\lambda\in\mathbb{R}: (-x,\lambda)\in\operatorname*{epi}\varphi$.
Since $(x,t)\in\operatorname*{epi}\varphi$ for all $t\in\mathbb{R}$ and $\varphi$ is subadditive, we get
$(x-x,t+\lambda)\in\operatorname*{epi}\varphi$ for all $t\in\mathbb{R}$. Thus $\varphi(0)=-\infty$, a contradiction.
\item[(c)] $\varphi(0)\in\mathbb{R} \Rightarrow (0,\varphi(0))\in\operatorname*{epi}\varphi$. $\Rightarrow (0,\varphi(0))+(0,\varphi(0))=(0,2\varphi(0))\in\operatorname*{epi}\varphi$.
$\Rightarrow \varphi(0)\le 2\varphi(0)$. $\Rightarrow \varphi(0)\ge 0$.\\
If $\varphi(x)=-\infty$, then $(x,t)\in\operatorname*{epi}\varphi$ for all $t\in\mathbb{R}$.
$\Rightarrow (nx,nt)\in\operatorname*{epi}\varphi$ for all $t\in\mathbb{R},n\in\mathbb{N}\setminus\{ 0\}$.
$\Rightarrow \varphi(nx)=-\infty$ for all $n\in\mathbb{N}\setminus\{ 0\}$.
\end{itemize}
\end{proof}

\begin{proposition}\label{p-add-fktal}
Let  $\varphi: X\rightarrow \overline{\mathbb{R}}_{\nu }$ be an additive function on a linear space $X$ with $0\in\operatorname*{dom}\varphi$. 
\begin{itemize}
\item[(a)] If $\varphi(0)\not=0$, then $\varphi$ does not attain any real value.
\item[(b)] If $\varphi(0)=0$, then $\varphi$ is odd on $\operatorname*{dom}\varphi\cap(-\operatorname*{dom}\varphi)$ and \\
$\varphi (tx)=t\varphi (x)$ for all $t\in\mathbb{Q}$, $x\in \operatorname*{dom}\varphi\cap(-\operatorname*{dom}\varphi)$.
\end{itemize}
\end{proposition}
\begin{proof}
From Lemma \ref{l-subadd} we get:\\
If $\varphi(0)\notin\mathbb{R}$, then $\varphi$ does not attain any real value.\\
If $\varphi(0)\in\mathbb{R}$, then $\varphi(0)=0$.\\
This implies (a).\\
Assume now $\varphi(0)=0$. Lemma \ref{l-subadd} implies for $x\in\operatorname*{dom}\varphi\cap(-\operatorname*{dom}\varphi)$:
$\varphi(x)=-\infty$ if and only if $\varphi(-x)=+\infty$.
Consider now some arbitrary $x\in\operatorname*{dom}\varphi\cap(-\operatorname*{dom}\varphi)$ with $\varphi(x)\in\mathbb{R}$.
$\Rightarrow (x,\varphi(x))$, $(-x,\varphi(-x))\in\operatorname*{epi}\varphi\cap\operatorname*{hypo}\varphi$.
$\Rightarrow (x,\varphi(x))+(-x,\varphi(-x))=(0,\varphi(x)+\varphi(-x))\in\operatorname*{epi}\varphi\cap\operatorname*{hypo}\varphi$.
$\Rightarrow 0=\varphi(0)=\varphi(x)+\varphi(-x)$. $\Rightarrow \varphi(-x)=-\varphi(x)$. Thus $\varphi$ is odd on $\operatorname*{dom}\varphi\cap(-\operatorname*{dom}\varphi)$.\\
Lemma \ref{l-subadd} implies $\varphi (nx)=n\varphi (x)$ for all $n\in\mathbb{N}$ and all $x\in \operatorname*{dom}\varphi$ with $\varphi(x)\notin\mathbb{R}$.
Consider now some $x\in \operatorname*{dom}\varphi$ with $\varphi(x)\in\mathbb{R}$.
$\Rightarrow (x,\varphi(x))\in \operatorname*{epi}\varphi\cap\operatorname*{hypo}\varphi$.
$\Rightarrow (nx,n\varphi(x))\in \operatorname*{epi}\varphi\cap\operatorname*{hypo}\varphi$.
$\Rightarrow \varphi (nx)=n\varphi(x)$ for all $n\in\mathbb{N}$. Thus $\varphi (nx)=n\varphi(x)$ for all $n\in\mathbb{N}$, $x\in\operatorname*{dom}\varphi$.
For $t\in -\mathbb{N}$, $x\in \operatorname*{dom}\varphi\cap(-\operatorname*{dom}\varphi)$ we get
$\varphi (tx)=\varphi ((-t)(-x))=-t\varphi(-x)=-t\cdot (-\varphi (x))=t\varphi (x)$.
Consider $q\in\mathbb{Z}\setminus\{0\}$, $x\in \operatorname*{dom}\varphi\cap(-\operatorname*{dom}\varphi)$. 
Then $\varphi (x)=\varphi (q\frac{1}{q}x)=q\varphi (\frac{1}{q}x)$.
$ \Rightarrow \varphi (\frac{1}{q}x)=\frac{1}{q}\varphi (x)$.
Consequently, $\varphi (\frac{p}{q}x)=\frac{p}{q}\varphi (x)$ for all $x\in \operatorname*{dom}\varphi, p\in \mathbb{N}, q\in\mathbb{Z}\setminus\{0\}$.
\end{proof}

\begin{theorem}\label{t-lin}
Let $X$ be a linear space and $\varphi :X \rightarrow \overline{\mathbb{R}}$. 
\begin{itemize}
\item[(1)] $\varphi$ is linear if and only if $\varphi$ is homogeneous and additive.
\item[(2)] $\varphi$ is linear if and only if $\varphi$ and $-\varphi$ are sublinear.
\end{itemize}
\end{theorem}
\begin{proof}
\begin{itemize}
\item[]
\item[(a)] We suppose that $\varphi$ is linear.\\ 
Assume that $\varphi$ is not odd. $\Rightarrow\exists x\in X: \varphi(-x)\not= -\varphi(x)$.
Proposition \ref{cx_odd_infty}(d) implies $\varphi(x)\not= -\infty$ and $\varphi(-x)\not= -\infty$. Applying Proposition \ref{cx_odd_infty}(d) to the convex function $-\varphi$, we get $\varphi(x)\not= +\infty$
and $\varphi(-x)\not= +\infty$. Thus $t_1:=\varphi(x)\in\mathbb{R}$, $t_2:=\varphi(-x)\in\mathbb{R}$. $\Rightarrow (x,t_1),(-x,t_2)\in\operatorname*{epi}\varphi\cap\operatorname*{hypo}\varphi$. $\Rightarrow 
(\frac{1}{2}(x-x),\frac{1}{2}(t_1+t_2))\in\operatorname*{epi}\varphi\cap\operatorname*{hypo}\varphi$. $\Rightarrow 0=\varphi(0)=\frac{1}{2}(t_1+t_2)$.
$\Rightarrow t_1=-t_2$, a contradiction. Hence $\varphi$ is odd.\\
If $\varphi(x)=-\infty$, then Proposition \ref{p-conv-infty}(b) implies $\varphi(\lambda x)=-\infty$ for all $\lambda\in (0,1)$.
If $\varphi(x)=+\infty$ and we apply Proposition \ref{p-conv-infty}(b) to $-\varphi$, we get $\varphi(\lambda x)=+\infty$ for all $\lambda\in (0,1)$.
Consider now some $x\in X$ with $t:=\varphi(x)\in\mathbb{R}$. $\Rightarrow (0,0), (x,t)\in\operatorname*{epi}\varphi\cap\operatorname*{hypo}\varphi$.
$\Rightarrow ((1-\lambda)\cdot 0+\lambda x,(1-\lambda)\cdot 0+\lambda t)\in\operatorname*{epi}\varphi\cap\operatorname*{hypo}\varphi$ for all $\lambda\in [0,1]$.
$\Rightarrow (\lambda x,\lambda t)\in\operatorname*{epi}\varphi\cap\operatorname*{hypo}\varphi$ for all $\lambda\in [0,1]$.
$\Rightarrow \varphi (\lambda x)=\lambda t$ for all $\lambda\in [0,1]$.
Thus $\varphi (\lambda x)=\lambda \varphi(x)$ for all $x\in X$, $\lambda\in [0,1]$.\\
Consider some $x\in X$ and $\lambda >1$. $\Rightarrow \frac{1}{\lambda}\in (0,1)$.
$\Rightarrow \varphi (x)=\varphi (\frac{1}{\lambda}(\lambda x))=\frac{1}{\lambda}\varphi (\lambda x)$.
$\Rightarrow \varphi (\lambda x)=\lambda \varphi (x)$.
Hence $\varphi (\lambda x)=\lambda \varphi (x)$ for all $\lambda\in \mathbb{R}_{+}$.
Since $\varphi$ is odd, we get $\varphi (\lambda x)=\lambda \varphi (x)$ for all $\lambda\in \mathbb{R}$.
Thus $\varphi$ is homogeneous by Proposition \ref{p-hom}.
\item[(b)] (a) implies part (2) of our assertion because of Lemma \ref{l-sublin}.
\item[(c)] Applying Lemma \ref{l-sublin} to $\varphi$ and $-\varphi$ implies the additivity of each linear function $\varphi$.
\item[(d)] If $\varphi$ is homogeneous and additive, then  $\varphi$ is linear by (2) and Lemma \ref{l-sublin}.
\end{itemize}
\end{proof}

Thus for real-valued functions, our definition of linearity coincides with the usual one.
\begin{corollary}
Let $X$ be a linear space and $\varphi: X \to \mathbb{R}$. \\
Then $\varphi$ is linear if and only if  
\[ 
\varphi(\lambda_1 x^1+\lambda_2 x^2) = \lambda_1 \varphi(x^1) + \lambda_2 \varphi(x^2)
\]
holds for all $x^1, x^2\in X$ and $\lambda_1 ,\lambda_2\in\mathbb{R}$.
\end{corollary}

\begin{theorem}\label{t-lin-add}
Let $X$ be a topological vector space and $\varphi :X \rightarrow \overline{\mathbb{R}}$ be continuous. 
Then $\varphi$ is linear if and only if $\varphi$ is additive and $\varphi (0)=0$.
\end{theorem}
\begin{proof}
Assume that $\varphi$ is additive and $\varphi (0)=0$, but that $\varphi$ is not linear.
Then $\varphi$ is not homogeneous by Theorem \ref{t-lin}.
$\Rightarrow \exists x^0\in X \;\exists \lambda\in\mathbb{R}: \varphi(\lambda x^0)\not=\lambda\varphi( x^0)$.
$\Rightarrow$ There exist neighborhoods $V_1$ of $\varphi(\lambda x^0)$ and $V_2$ of $\lambda\varphi( x^0)$ such that $V_1\cap V_2=\emptyset$.
Since $\varphi$ is continuous, the function $\varphi(\lambda x)$ is a continuous function of $x$.
Thus there exists some neighborhood $U$ of $x^0$ such that for each $x\in U$: $\varphi(\lambda x)\in V_1$.
There exists some $p\in(0,1)$ such that for all $t\in(p,1]$: $tx^0\in U$.
Consider some sequence $(q_n)_{n\in\mathbb{N}}$ of rational numbers with $q_n<\lambda$ for each $n\in\mathbb{N}$ which converges to $\lambda$.
$\Rightarrow \exists n_0\in\mathbb{N}\;\forall n>n_0: \frac{q_n}{\lambda }\in(p,1]$.
$\Rightarrow \;\forall n>n_0: \frac{q_n}{\lambda }x^0\in U$.
$\Rightarrow \;\forall n>n_0: q_n\varphi(x^0)=\varphi(q_n x^0)=\varphi(\lambda \frac{q_n}{\lambda } x^0)\in V_1$.
$\Rightarrow \;\forall n>n_0: q_n\varphi(x^0)\not\in V_2$.
$\Rightarrow q_n\varphi(x^0)$ does not converge to $\lambda\varphi(x^0)$ for $n\rightarrow \infty  $, a contradiction.
\end{proof}
\smallskip

\section{Usage of Results from Convex Analysis}\label{s-trans-cx-ua}

Many essential facts for extended real-valued functions have been proved in convex analysis and can be applied to the extended real-valued functions in our approach by making use of the following remarks.

There is a one-to-one correspondence between extended real-valued functions in convex analysis and those extended real-valued functions in the unified approach which do not attain the value $+\infty$. This will be shown in Section \ref{s-corr-cx-ua}. In Section \ref{s-beyond-cx}, the way of transferring results from convex analysis  to functions in the unified approach which can also attain the value $+\infty$ is pointed out. Since we have not defined all properties which are of interest for functionals, we start in Section \ref{s-ext-prop-def} with an explanation of how the definition of notions and properties for functions works in the unified approach.

\subsection{Definition of notions and properties for functions in the unified approach}\label{s-ext-prop-def}

$ \quad $\\
By defining notions and properties in the unified approach, one simply has to take into consideration that the value $\nu$ stands for ''not being defined" or ''not being feasible". 

A function in the unified approach has a property if and only if it has it on its domain, where the domain has to fulfill those conditions which are essential for the defined property. The definition for the property on the domain can be taken from convex analysis. This is illustrated by the previous sections and by the following definitions where the second one is, e.g., of importance in vector optimization.

\begin{definition}
Let $X$ be a real normed space, $\varphi :X \rightarrow \overline{\mathbb{R}}_{\nu}$ and $X_0\subseteq\operatorname*{dom}\varphi$ with $X_0\not=\emptyset$. \\
$\varphi$ is called Lipschitz continuous on $X_0$ if $\varphi$ is finite-valued on $X_0$ and if there exists some $L\in\mathbb{R}_+$ such that
$|\varphi (x^1) -\varphi (x^2)| \leq L \| x^1-x^2\|$ for all $x^1,x^2\in X_0$.\\
$\varphi$ is said to be locally Lipschitz continuous on $X_0$ if for each $x\in X_0$ there exists some neighborhood $U$ of $x$ such that
$\varphi$ is Lipschitz continuous on $U\cap X_0$.
$\varphi$ is locally Lipschitz continuous or Lipschitz continuous if $\varphi$ is a proper functional which is locally Lipschitz continuous or Lipschitz continuous, respectively, on $\operatorname*{dom}\varphi$.
\end{definition}

\begin{definition}\label{def-Bmon}
Let $X$ be a linear space, $B\subseteq X$  and $\varphi: X \to \overline{\mathbb{R}}_{\nu }$, $X_0\subseteq \operatorname*{dom}\varphi$.\\
$\varphi$ is said to be
\begin{itemize}
\item[(a)]
$B$-monotone on $X_0$
if $x^1,x^2 \in X_0$ and $x^{2}-x^{1}\in B$ imply $\varphi
(x^{1})\le \varphi (x^{2})$,
\item[(b)] strictly $B$-monotone on $X_0$ 
if $x^1,x^2 \in X_0$ and $x^{2}-x^{1}\in B\setminus
\{0\}$ imply $\varphi (x^{1})<\varphi (x^{2})$.
\end{itemize}
$\varphi$ is said to be $B$-monotone or strictly $B$-monotone if it is $B$-monotone or strictly $B$-monotone, respectively, on $\operatorname*{dom}\varphi$.
\end{definition}

In variational analysis, basic definitions can be extended to the unified approach as follows.

\begin{definition}
Suppose that $X$ is a separated locally convex space with topological dual space $X^{\ast}$ and  $f:X\rightarrow\overline{\mathbb{R}}_{\nu }$.
The \textbf{(Fenchel) sub\-differential} $\partial\! f(x^0)$ of $f$ at $x^0 \in X$ is the set
$$\{x^{\ast}\in X^{\ast} \; \mid \; \forall x\in \operatorname*{dom}f: \; x^{\ast }(x-x^0 ) \leq f(x)-f(x^0)\ \}\mbox{ if } f(x^0)\in\mathbb{R},$$
and the empty set if $f(x^0)\notin\mathbb{R}$.\\
The \textbf{(Fenchel) conjugate} of $f$ is the function $f^{\ast}:X^{\ast}\rightarrow \overline{\mathbb{R}}_{\nu }$ defined by
\[ f^{\ast}(x^{\ast}):=\sup\left\{ x^{\ast} (x) -f(x) \; \mid \; x\in \operatorname*{dom}f\right\}.\]
The \textbf{support function} $\sigma_{A}:X^{\ast}\rightarrow \overline{\mathbb{R}}_{\nu }$ of a set $A\subseteq X$ is defined by
\[ \sigma_{A}(x^{\ast}):=\sup\left\{ x^{\ast} (a) \; \mid \; a\in A\right\}.\]
The set $\{x^{\ast}\in X^{\ast}\mid \sigma_{A}(x^{\ast})\in\mathbb{R}\}$ is called the \textbf{barrier cone} $\operatorname*{bar}A$ of $A$.\\
\end{definition}
\smallskip

\subsection{Usage of results from convex analysis for functions which do not attain the value $+\infty$ in the unified approach}\label{s-corr-cx-ua}

$ \quad $\\
If a functional attains only values from $\mathbb{R}\cup\{-\infty\}$ on the entire space, then the notions for and properties of this function are the same in convex analysis and in our unified approach. 

If a statement in convex analysis is given for some functional $f_{cxa}:X\to\overline{\mathbb{R}}$, we can transfer this statement to $f_{ua}:X\to\overline{\mathbb{R}}_{\nu }$ in the unified approach given by
\begin{equation*}
f_{ua}(x)=\left\{
\begin{array}{c@{\quad\mbox{ if }\quad} l}
f_{cxa}(x) & f_{cxa}(x)\not=+\infty,\\
\nu & f_{cxa}(x)=+\infty ,
\end{array}
\right.
\end{equation*}
using the following interdependencies.\\
Note that $f_{ua}$ does not attain the value $+\infty$. 

\begin{itemize}
\item[I.] {\bf Notions and properties which are similar for $f_{cxa}$ and $f_{ua}$}\smallskip
\begin{itemize}
\item[(a)] The (effective) domain of $f_{cxa}$ in the terminology of convex analysis is, in the unified approach, the set
\begin{eqnarray*}
\qquad\operatorname*{dom}f_{ua} & = & \{x\in X\mid f_{ua}(x)\in\mathbb{R}\cup\{-\infty\}\}\\
& = & \{x\in X\mid f_{cxa}(x)\in\mathbb{R}\cup\{-\infty\}\}= \operatorname*{dom}\nolimits_{-}f_{cxa}.
\end{eqnarray*}
\item[(b)] $f_{ua}$ is proper in the terminology of the unified approach\\
$\Leftrightarrow f_{cxa}$ is proper in the terminology of convex analysis\\
$\Leftrightarrow f_{ua}$ attains some real value, but not the value $-\infty$ on $X$\\
$\Leftrightarrow f_{cxa}$ attains some real value, but not the value $-\infty$ on $X$.
\item[(c)] $f_{ua}$ is finite-valued in the terminology of the unified approach\\
$\Leftrightarrow f_{cxa}$ is finite-valued in the terminology of convex analysis\\
$\Leftrightarrow f_{ua}$ attains only real values on $X$\\ 
$\Leftrightarrow f_{cxa}$ attains only real values on $X$.
\item[(d)] $\operatorname*{epi}f_{ua}=\operatorname*{epi}f_{cxa}$.\\
Our definition of the epigraph applied to $f_{cxa}$ coincides with that in convex analysis.
\item[(e)] For each $x\in\operatorname*{dom}_{-}f_{cxa}$, $f_{cxa}$ is lower semicontinuous at $x$ if and only if $f_{ua}$ is lower semicontinuous at $x$.\\
The notion of lower semicontinuity of $f_{cxa}$ at $x\in\operatorname*{dom}_{-}f_{cxa}$ coincides in convex analysis and in the unified approach.
\item[(f)] $f_{cxa}$ is convex if and only if $f_{ua}$ is convex.\\
Our definition of convexity applied to $f_{cxa}$ coincides with that in convex analysis.
\item[(g)]  For $f_{cxa}$, the definitions of $B$-monotonicity, strict $B$-monotonicity, Lipschitz continuity and local Lipschitz continuity on a nonempty subset of
$\operatorname*{dom}_{-}f_{cxa}$ in the unified approach are the same as in convex ana\-lysis. $f_{cxa}$ has one of these properties if and only if $f_{ua}$ has it.
\item[(h)] $\partial\! f_{ua}(x)=\partial\! f_{cxa}(x)$ for each $x\in X$ if $X$ is a separated locally convex space.\\
Our definition of the subdifferential applied to $f_{cxa}$ coincides with that in convex analysis.\\
\end{itemize}
\item[II.] {\bf Notions and properties which are different for $f_{cxa}$ and $f_{ua}$}\smallskip
\begin{itemize}
\item[(j)] For each nonempty set $X_0\subseteq \operatorname*{dom}f_{ua}$, the notions infimum, supremum, maximum, minimum, upper bound and lower bound  on $X_0$ have the same contents for $f_{cxa}$ and $f_{ua}$. In this case, the definitions of these notions coincide for $f_{cxa}$ with the usual definitions in convex analysis.\\
If $X_0=\emptyset$, we get in convex analysis,
$\operatorname*{inf}_{x\in X_0}f_{cxa}(x)= +\infty$ and 
$\operatorname*{sup}_{x\in X_0}f_{cxa}(x)= -\infty$,\\
and in the unified approach,
$\operatorname*{inf}_{x\in X_0}f_{ua}(x)=\operatorname*{sup}_{x\in X_0}f_{ua}(x)=\nu $.\\
If $X_0\setminus\operatorname*{dom}_{-}f_{cxa}\not=\emptyset$, we get
in convex analysis,\\
$\operatorname*{inf}_{x\in X_0}f_{cxa}(x)= \operatorname*{inf}_{x\in\operatorname*{dom}f_{cxa}}f_{cxa}(x)$ and 
$\operatorname*{sup}_{x\in X_0}f_{cxa}(x)= +\infty$,
and in the unified approach,
$\operatorname*{inf}_{x\in X_0}f_{ua}(x)=\operatorname*{inf}_{x\in\operatorname*{dom}f_{ua}}f_{ua}(x)$\\
and $\operatorname*{sup}_{x\in X_0}f_{ua}(x)= \operatorname*{sup}_{x\in\operatorname*{dom}f_{ua}}f_{ua}(x)$. 
\item[(k)] $f_{cxa}$ is continuous at $x\in\operatorname*{dom}_{-}f_{cxa}$ if and only if $x\in\operatorname*{int}\operatorname*{dom}f_{ua}$ and $f_{ua}$ is continuous at $x$.\\
The analogous statement for upper semicontinuity instead of continuity holds as well.\\
At each $x\in\operatorname*{dom}_{-}f_{cxa}\setminus \operatorname*{int}(\operatorname*{dom}_{-}f_{cxa})$, $f_{cxa}$ is not upper semicontinuous and hence not continuous according to the definition in convex analysis as well as to the definition in the unified approach, but $f_{ua}$ may be upper semicontinuous or continuous at $x$ in the unified approach.\\
At $x\in\operatorname*{int}(\operatorname*{dom}_{-}f_{cxa})$, the notions of continuity and upper semicontinuity of $f_{cxa}$ in convex analysis are the same as in the unified approach.
\end{itemize}
\smallskip
Let us point out that these differences are advantages of the unified approach.
\end{itemize}

\subsection{Usage of results from convex analysis for functions which attain the value $+\infty$ in the unified approach}\label{s-beyond-cx}

$ \quad $\\
Statements from convex analysis for functionals $f_{cxa}:X\to\overline{\mathbb{R}}$ can be transferred to statements
for functions $\varphi :X\to\overline{\mathbb{R}}_{\nu}$ in the unified approach which can also attain the value $+\infty$ according to the following rules.\\

\begin{tabular}{|c|c|}\hline
Replace & by\\ \hline\hline
$X$ & $\operatorname*{dom}\varphi $\\ \hline
$\operatorname*{dom}f_{cxa}$ & $\operatorname*{dom}_{-}\varphi $\\ \hline
finite-valued & proper\\ \hline
proper & $\varphi$ attains some real value and not the value $-\infty$\\ \hline
\end{tabular}\\ 
\smallskip

The sets on which $f_{cxa}$ has a certain property have to be adapted to $\varphi$ according to this table. If the set is $X$, those conditions which the whole space $X$ automatically fulfills have to be taken into consideration.
\smallskip

\section{Final Remarks}
Our theory serves as a basis for extending functions with values in $\overline{\mathbb{R}}$ to the entire space and handling them in a way which is not unilateral. It offers the same possibilities as the approach in convex analysis, but avoids several of its disadvantages and delivers a calculus for functions not depending on the purpose they are used for.

All results of the unified approach can immediately be applied in the classical framework of convex analysis since each function $\varphi : X\to\overline{\mathbb{R}}$ also maps into $\overline{\mathbb{R}}_{\nu }$. 

\smallskip
\begin{tabular}{|c|c|}\hline
Replace (unified approach) & by (classical approach) \\ \hline\hline
$\operatorname*{dom}\varphi$ & $X$\\ \hline
$\operatorname*{dom}_{-}\varphi$ & $\operatorname*{dom}\varphi$\\ \hline
proper & finite-valued\\ \hline
\end{tabular}\\ 
\smallskip

Let us finally mention that the application of $\nu$ is not restricted to $\overline{\mathbb{R}}$. $\nu$ can be considered in combination with each space $Y$. The concept presented in this paper can easily be adapted to $Y_{\nu }:=Y\cup\{\nu \}$ and to functions with values in $Y_{\nu }$.
Suppose some function $f:C\rightarrow Y$ with $C$ being a subset of some space $X$ and $Y$ being a linear space. 
$f$ could be extended to $X$ by the function value $\nu$. 
An indicator function $\iota_{C}:X\rightarrow Y_{\nu}$ can be defined in the same way as above, where $0$ stands for the element of $Y$ which is neutral w.r.t. addition.

\end{document}